\documentclass[11pt]{amsart}
\usepackage{amsmath,amssymb,graphicx,setspace,verbatim}
\usepackage[pagewise]{lineno}
\usepackage{color}
\usepackage{comment}
\usepackage{appendix}
\usepackage[margin=1.5in,footskip=0.2in]{geometry}
\theoremstyle{comment}

\newcommand{\PGL}{\mathrm{PGL}}

\newcommand{\Cyc}{\mathrm{C}} 
\newcommand{\Alt}{\mathrm{A}} 
\newcommand{\Sym}{\mathrm{S}}

\newtheorem*{mcomment}{\color{cyan}{Comment}}

\input xy
\xyoption{all}
\numberwithin{equation}{section}

\newtheorem{theorem}{Theorem}[section]
\newtheorem{lemma}[theorem]{Lemma}
\theoremstyle{definition}
\newtheorem{definition}[theorem]{Definition}
\newtheorem{proposition}[theorem]{Proposition}
\newtheorem{corollary}[theorem]{Corollary}

\begin{document}

\title{An extension of the classification of high rank regular polytopes}

\author{Maria Elisa Fernandes}
\address{
Maria Elisa Fernandes, Center for Research and Development in Mathematics and Applications, Department of Mathematics, University of Aveiro, Portugal
}
\email{maria.elisa@ua.pt}

\author{Dimitri Leemans}
\address{Dimitri Leemans, Department of Mathematics, University of Auckland, Private Bag 92019, Auckland 1142, New Zealand
}
\email{d.leemans@auckland.ac.nz}

\author{Mark Mixer}
\address{Mark Mixer, Department of Applied Mathematics, Wentworth Institute of Technology, Boston, MA 02115, USA
}
\email{mixerm@wit.edu}

\begin{abstract}
Up to isomorphism and duality, there are exactly two non-degenerate abstract regular polytopes of rank greater than $n-3$, one of rank $n-1$ and one of rank $n-2$, with automorphism groups that are transitive permutation groups of degree $n\geq 7$. In this paper we extend this classification of high rank regular polytopes to include the ranks $n-3$ and $n-4$.  The result is, up to a isomorphism and duality, seven abstract regular polytopes of rank $n-3$ for each $n\geq 9$, and nine abstract regular polytopes of rank $n-4$ for each $n \geq 11$. Moreover we show that if a transitive permutation group $\Gamma$ of degree $n \geq 11$ is the automorphism group of an abstract regular polytope of rank at least $n-4$, then $\Gamma\cong S_n$.
\end{abstract}
\maketitle

\noindent \textbf{2000 Math Subj. Class:} 52B11, 20D06.

\section{Introduction}
Abstract polytopes are incidence structures generalizing certain discrete geometric objects, such as the Platonic solids. Abstract regular polytopes are those that are richest in symmetry.
The one-to-one correspondence between abstract regular polytopes and a class of groups known as string C-groups, which are themselves quotients of Coxeter groups, has led to many papers which study when abstract regular polytopes have certain groups as their group of automorphisms.

Much of the work in this area has been influenced by databases of examples of polytopes and their automorphism groups.
In 2006, D. Leemans and L. Vauthier published ``An atlas of polytopes for almost simple groups,"~\cite{Lalg} classifying all abstract regular polytopes whose automorphism group is an almost simple group as large as the automorphism group of a simple group with 900,000 elements.
Also in 2006, M. Hartley published ``An atlas of small regular polytopes,"~\cite{Halg} where he classified all regular polytopes with automorphism groups of order at most 2000 (not including orders 1024 and 1536). In~\cite{HHalg},
M. Hartley and A. Hulpke classified all regular polytopes for the sporadic groups as large as the Held group and Leemans and Mixer classified, among others, all regular polytopes for the third Conway group ~\cite{LMAlg}.
More recently, T. Connor, Leemans and Mixer classified all regular polytopes of rank at least four for the O'Nan sporadic simple group~\cite{CLM2014}.

These collections of polytopes helped lead to various theoretical results about abstract regular polytopes with automorphism group $G$ an almost simple group, where $PSL(2,q) \leq G \leq P\Gamma L(2,q)$, or $G$ is one of $PSL(3,q)$, $PGL(3,q)$, $Sz(q)$, $R(q)$, $S_n$, or $A_n$ (see ~\cite{ls09, BV10, Leemans06, KL2010, Leemans:2015, fl, flm, flm2, CDL2013}). Table~\ref{sn} summarizes the results obtained for the symmetric groups $S_n$.  The results for $5\leq n\leq 9$ can be found in~\cite{Lalg}, and the regular polytopes for larger $S_n$ are computed using the algorithms developed in~\cite{LMAlg}.

Julius Whiston showed in~\cite{Whis00} that the maximum size of an independent generating set of a permutation group of degree $n$ is $n-1$.
Thus the rank of an abstract regular polytope whose automorphism group is a permutation group of degree $n$ is at most $n-1$.
Looking at Table~\ref{sn}, Fernandes and Leemans proved in~\cite{fl}
that there is exactly one such polytope up to isomorphism when $n\geq 5$, and
that when $n\geq 7$, there is also a unique polytope of rank $n-2$ up to isomorphism and duality, and that for each rank $3\leq r \leq n-1$, there is at least one abstract regular polytope of rank $r$ whose automorphism group is $S_n$.

\begin{table}
\begin{center}
\begin{tabular}{||c|c|c|c|c|c|c|c|c|c|c|c||}
\hline
$\textbf{G}\backslash \textbf{Rank}$&\textbf{3}&\textbf{4}&\textbf{5}&\textbf{6}&\textbf{7}&\textbf{8}&\textbf{9}&\textbf{10}&\textbf{11}&\textbf{12}&\textbf{13}\\
\hline
$S_5$&4&1&0&0&0&0&0&0&0&0&0\\
$S_6$&2&4&1&0&0&0&0&0&0&0&0\\
$S_7$&35&7&1&1&0&0&0&0&0&0&0\\
$S_8$&68&36&11&1&1&0&0&0&0&0&0\\
$S_9$&129&37&7&7&1&1&0&0&0&0&0\\
$S_{10}$&413 &203 &52 &13 &7 &1 &1 &0&0&0&0\\
$S_{11}$&1221 &189 &43 &25 &9 &7 &1 &1&0&0&0\\
$S_{12}$&3346 &940 &183 &75 &40 &9 &7 &1 &1 &0&0\\
$S_{13}$&7163 &863 &171 &123 &41 &35 &9 &7 &1 &1&0\\
$S_{14}$&23126 &3945 &978 &303 &163 &54 &35 &9 &7&1&1\\
\hline
\end{tabular}
\begin{caption}{Number of polytopes for $S_n$ ($5\leq n \leq 14$). }\label{sn} \end{caption}
\end{center}
\end{table}

In this article, we extend the results  of ~\cite{fl} by giving a classification of rank $r \geq n-4$ string C-groups with connected diagram for transitive groups of degree $n$ with $n$ sufficiently large. More precisely, the main theorem of this article is the following.

\begin{theorem}\label{maintheorem}
Let $1\geq i \geq 4$, and $n\geq 3+2i$ when $r = n-i$.
If $\Gamma$ is a string C-group of rank $r\geq n-i$ with a connected diagram and is isomorphic to a transitive group of degree $n$  then $\Gamma$ or its dual is isomorphic to $S_n$ and the CPR graph is one of those listed in Figure~\ref{7cases}.
\end{theorem}

\begin{small}
\begin{figure}[h]
\label{7cases}
\begin{tabular}{cc}
\begin{tabular}{c}
\begin{tabular}{|c|}
\hline
$ \xymatrix@-1pc{ *+[o][F]{}    \ar@{-}[r]^0    &  *+[o][F]{}   \ar@{-}[r]^1  & *+[o][F]{}  \ar@{-}[r]^2 & *+[o][F]{}  \ar@{-}[r]^3 & *+[o][F]{} \ar@{.}[r] &  *+[o][F]{}  \ar@{-}[r]^{n-2} & *+[o][F]{}}$\\
\hline
\end{tabular}\\

\\

\begin{tabular}{|c|}
\hline
$ \xymatrix@-1pc{ *+[o][F]{}   \ar@{-}[r]^1  & *+[o][F]{}  \ar@{-}[r]^0    &  *+[o][F]{}   \ar@{-}[r]^1  & *+[o][F]{}  \ar@{-}[r]^2 & *+[o][F]{}  \ar@{-}[r]^3 & *+[o][F]{} \ar@{.}[r] &   *+[o][F]{}  \ar@{-}[r]^{n-4} & *+[o][F]{}  \ar@{-}[r]^{n-3} & *+[o][F]{} }$\\
\hline
\end{tabular}\\
\\
\begin{tabular}{|c|}
\hline
$ \xymatrix@-1pc{
     *+[o][F]{}   \ar@{-}[r]^2  \ar@{-}[d]_0     & *+[o][F]{}  \ar@{-}[r]^1 \ar@{-}[d]^0  & *+[o][F]{}  \ar@{-}[r]^2 & *+[o][F]{}  \ar@{-}[r]^3 & *+[o][F]{}  \ar@{-}[r]^4 & *+[o][F]{}  \ar@{..}[r] & *+[o][F]{}  \ar@{-}[r]^{n-4}& *+[o][F]{}  \\
     *+[o][F]{}  \ar@{-}[r]_2     & *+[o][F]{}    &  & & & & &  }$
\\
$ \xymatrix@-1pc{
     *+[o][F]{}   \ar@{-}[r]^0  \ar@{=}[d]_1^2     & *+[o][F]{}  \ar@{-}[r]^1 \ar@{-}[d]^2  & *+[o][F]{} \ar@{-}[r]^2 & *+[o][F]{}  \ar@{-}[r]^3 & *+[o][F]{}  \ar@{-}[r]^4 & *+[o][F]{}  \ar@{..}[r] & *+[o][F]{}  \ar@{-}[r]^{n-4}& *+[o][F]{}  \\
     *+[o][F]{}  \ar@{-}[r]_0     & *+[o][F]{}    &  & & & & &  }$
     \\
$ \xymatrix@-1pc{ *+[o][F]{}   \ar@{-}[r]^1  & *+[o][F]{}  \ar@{-}[r]^0    &  *+[o][F]{}   \ar@{-}[r]^1  & *+[o][F]{}  \ar@{-}[r]^2 & *+[o][F]{}  \ar@{-}[r]^3 & *+[o][F]{} \ar@{.}[r] &  *+[o][F]{}  \ar@{-}[r]^{n-6} & *+[o][F]{}  \ar@{-}[r]^{n-5} & *+[o][F]{}  \ar@{-}[r]^{n-4} & *+[o][F]{}  \ar@{-}[r]^{n-5} & *+[o][F]{}}$
\\
 $ \xymatrix@-1pc{ *+[o][F]{}   \ar@{-}[r]^2  & *+[o][F]{}  \ar@{-}[r]^1    &  *+[o][F]{}   \ar@{-}[r]^0  & *+[o][F]{}  \ar@{-}[r]^1 & *+[o][F]{}  \ar@{-}[r]^2 & *+[o][F]{}  \ar@{-}[r]^3  & *+[o][F]{}  \ar@{-}[r]^4  &*+[o][F]{} \ar@{.}[r] &  *+[o][F]{}  \ar@{-}[r]^{n-5} & *+[o][F]{}  \ar@{-}[r]^{n-4} & *+[o][F]{}  }$
 \\
$ \xymatrix@-1pc{ *+[o][F]{}   \ar@{-}[r]^0  & *+[o][F]{} \ar@{-}[r]^1  &  *+[o][F]{} \ar@{-}[r]^0 & *+[o][F]{} \ar@{-}[r]^1& *+[o][F]{} \ar@{-}[r]^2  & *+[o][F]{}  \ar@{-}[r]^3  & *+[o][F]{}  \ar@{-}[r]^4 & *+[o][F]{} \ar@{.}[r] & *+[o][F]{}  \ar@{-}[r]^{n-5} & *+[o][F]{}  \ar@{-}[r]^{n-4} & *+[o][F]{}}$
\\
 $ \xymatrix@-1pc{ *+[o][F]{}   \ar@{-}[r]^0  & *+[o][F]{} \ar@{-}[r]^1  &  *+[o][F]{} \ar@{=}[r]^0_2 & *+[o][F]{} \ar@{-}[r]^1& *+[o][F]{} \ar@{-}[r]^2  & *+[o][F]{}  \ar@{-}[r]^3  & *+[o][F]{}  \ar@{-}[r]^4 & *+[o][F]{} \ar@{.}[r] & *+[o][F]{}  \ar@{-}[r]^{n-5} & *+[o][F]{}  \ar@{-}[r]^{n-4} & *+[o][F]{}}$
 \\
$ \xymatrix@-1pc{ *+[o][F]{}   \ar@{=}[r]^2_0  & *+[o][F]{} \ar@{-}[r]^1  &  *+[o][F]{} \ar@{-}[r]^0 & *+[o][F]{} \ar@{-}[r]^1& *+[o][F]{} \ar@{-}[r]^2  & *+[o][F]{}  \ar@{-}[r]^3  & *+[o][F]{}  \ar@{-}[r]^4 & *+[o][F]{} \ar@{.}[r] & *+[o][F]{}  \ar@{-}[r]^{n-5} & *+[o][F]{}  \ar@{-}[r]^{n-4} & *+[o][F]{}}$  \\
\hline
\end{tabular}
\end{tabular}
&
\begin{tabular}{|c|}
\hline
$\xymatrix@-1pc{
         *+[o][F]{}   \ar@{-}[d]_3  \ar@{-}[r]^0     & *+[o][F]{}  \ar@{-}[d]^3  \ar@{-}[r]^1  & *+[o][F]{} \ar@{-}[d]^3 \ar@{-}[r]^2 & *+[o][F]{} \ar@{-}[r]^3 & *+[o][F]{} \ar@{-}[r]^4& *+[o][F]{} \ar@{.}[r] &*+[o][F]{} \ar@{-}[r]^{n-5}&*+[o][F]{}\\
      *+[o][F]{}  \ar@{-}[r]_0     & *+[o][F]{}  \ar@{-}[r]_1   &  *+[o][F]{}&  &&&&}$
\\
$\xymatrix@-1pc{
      *+[o][F]{}   \ar@{-}[d]_2  \ar@{-}[r]^0     & *+[o][F]{}  \ar@{-}[d]^2  \ar@{-}[r]^1  & *+[o][F]{}\ar@{-}[r]^2 & *+[o][F]{}\ar@{-}[r]^3 & *+[o][F]{}\ar@{.}[r]& *+[o][F]{}\ar@{-}[r]^{n-5}&*+[o][F]{}\ar@{-}[r]^{n-6} &*+[o][F]{}\\
      *+[o][F]{}  \ar@{-}[r]_0     & *+[o][F]{}    &   &&&&&&&}$
\\
$\xymatrix@-1pc{
     *+[o][F]{}   \ar@{-}[r]^0  \ar@{=}[d]_2^1     & *+[o][F]{}  \ar@{-}[r]^1 \ar@{-}[d]^2 & *+[o][F]{}  \ar@{.}[r] &*+[o][F]{}  \ar@{-}[r]^{n-5} &*+[o][F]{}  \ar@{-}[r]^{n-6} &*+[o][F]{}\\
     *+[o][F]{}  \ar@{-}[r]_0     & *+[o][F]{}    & &&&   }$
\\
$\xymatrix@-1pc{ *+[o][F]{}   \ar@{-}[r]^0 & *+[o][F]{} \ar@{-}[r]^1  &  *+[o][F]{} \ar@{-}[r]^0 & *+[o][F]{} \ar@{.}[r] & *+[o][F]{}  \ar@{-}[r]^{n-5}& *+[o][F]{}  \ar@{-}[r]^{n-6}  & *+[o][F]{}}$
\\
$\xymatrix@-1pc{ *+[o][F]{}   \ar@{-}[r]^2  & *+[o][F]{} \ar@{-}[r]^1  &  *+[o][F]{} \ar@{-}[r]^0 & *+[o][F]{} \ar@{.}[r] & *+[o][F]{}  \ar@{-}[r]^{n-5}& *+[o][F]{}  \ar@{-}[r]^{n-6}  & *+[o][F]{}}$
\\
$\xymatrix@-1pc{ *+[o][F]{}   \ar@{=}[r]^0_2  & *+[o][F]{} \ar@{-}[r]^1  &  *+[o][F]{} \ar@{-}[r]^0 & *+[o][F]{} \ar@{.}[r] & *+[o][F]{}  \ar@{-}[r]^{n-5}& *+[o][F]{}  \ar@{-}[r]^{n-6}  & *+[o][F]{}}$
\\
$\xymatrix@-1pc{ *+[o][F]{}  \ar@{-}[r]^0& *+[o][F]{}  \ar@{-}[r]^{1} & *+[o][F]{}  \ar@{=}[r]^{2}_0 & *+[o][F]{}  \ar@{-}[r]^{1}  & *+[o][F]{} \ar@{.}[r]& *+[o][F]{}  \ar@{-}[r]^{n-5}  & *+[o][F]{}  \ar@{-}[r]^{n-6} & *+[o][F]{}}$
\\
$\xymatrix@-1pc{ *+[o][F]{}   \ar@{-}[r]^1  &*+[o][F]{}   \ar@{-}[r]^0  & *+[o][F]{} \ar@{-}[r]^1  &  *+[o][F]{} \ar@{-}[r]^0  & *+[o][F]{} \ar@{.}[r] &  *+[o][F]{}  \ar@{-}[r]^{n-5} & *+[o][F]{}}$
\\
$\xymatrix@-1pc{ *+[o][F]{}   \ar@{-}[r]^3 &*+[o][F]{}   \ar@{-}[r]^2  & *+[o][F]{} \ar@{-}[r]^1  &  *+[o][F]{} \ar@{-}[r]^0& *+[o][F]{}\ar@{.}[r] & *+[o][F]{}  \ar@{-}[r]^{n-5} & *+[o][F]{}}$
\\
\hline
\end{tabular}
\end{tabular}
\caption{CPR graphs for  string C-groups of rank $r\geq n-4$}
\end{figure}
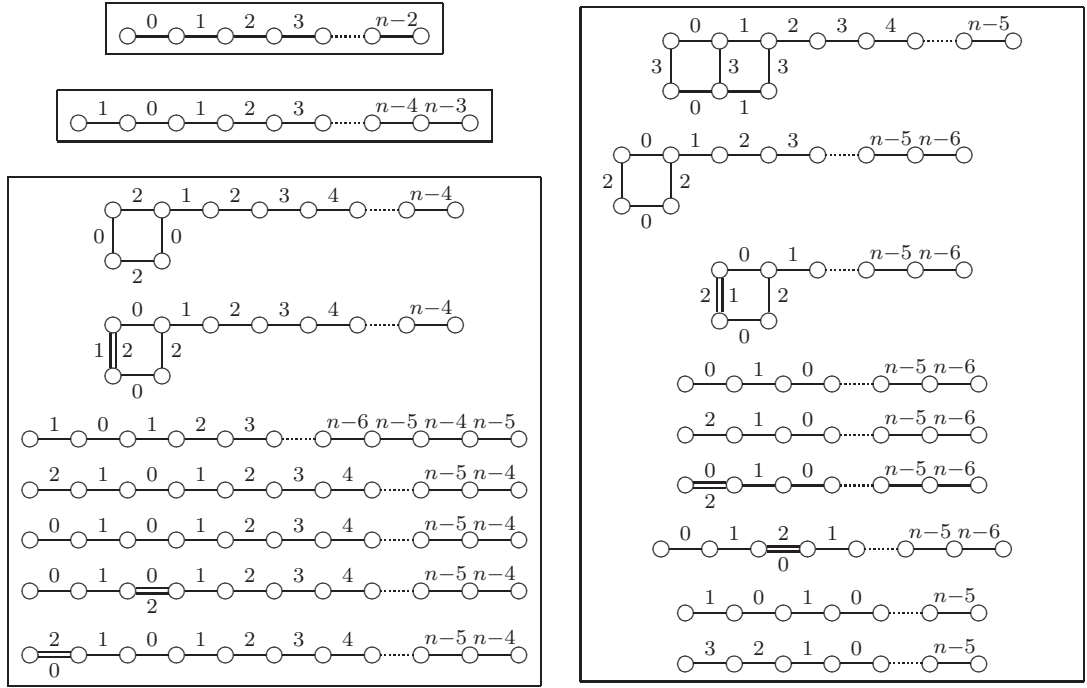
\end{small}

The cases where $r=n-1$ and $r=n-2$ were already dealt with in~\cite{fl}.
Here we improve the techniques used in ~\cite{fl} in the case where all maximal parabolic subgroups of $\Gamma$ are intransitive and we use the main theorem of \cite{CFLM2014} in the case where a maximal parabolic subgroup of $\Gamma$ is transitive. This makes cases $r=n-1$ and $r=n-2$ much easier to prove and it is the reason why we decide to include these cases in the present article as well.
We conjecture that similar results could be obtained for ranks $n-k$ with $4<k < \frac{n}{2}$ for $n$ sufficiently large. 

The article is organized as follows.
In Section~\ref{prelim}, we give the definitions and notation needed to understand this article.
In Section~\ref{sectionintrans}, we use {\em fracture graphs} to classify the possible permutation representation graphs of a rank $n-3$ or $n-4$ string group generated by involutions with only intransitive maximal parabolic subgroups.
In Section~\ref{section6}, we prove that exactly seven of these possibilities yield a rank $n-3$ string C-group and exactly nine yield a rank $n-4$ string C-group.
Finally, in Section~\ref{sectiontrans}, we show that only the cases appearing in Section~\ref{section6} may occur by proving that, under the hypotheses of Theorem~\ref{maintheorem}, all maximal parabolic subgroups must be intransitive.

\section{Preliminares}\label{prelim}
\subsection{String C-groups}
A \emph{Coxeter group} $\Gamma := (G,S)$ of rank $r$ is a group $G$ with a set $S$ of distinguished generators $\rho_0,\ldots,\rho_{r-1}$ and presentation
$$\langle \rho_i \ | \ (\rho_i\rho_j)^{m_{i,j}}=\varepsilon\mbox{ for all }i,\,j\in\{0,\ldots,r-1\} \rangle$$
where $\varepsilon$ is the identity element of $G$ and each $m_{i,j}$ is a positive integer or infinity, $m_{i,i}=1$, and $m_{i,j}>1$ for $i\neq j$.  It follows from the definition, that a Coxeter group satisfies the next condition which is called the \emph{intersection property}.
\[\forall J, K \subseteq \{0,\ldots,r-1\}, \langle \rho_j \mid j \in J\rangle \cap \langle \rho_k \mid k \in K\rangle = \langle \rho_j \mid j \in J\cap K\rangle\]
A Coxeter group $\Gamma$ can be represented by a \emph{Coxeter diagram} $\mathcal{D}$. This Coxeter diagram $\mathcal{D}$ is a labelled graph which represents the set of relations of $G$. More precisely, the vertices of the graph correspond to the generators $\rho_i$ of $G$, and for each $i$ and $j$, an edge with label $m_{i,j}$ joins the $i$th and the $j$th vertices; conventionally, edges of label 2 are omitted. By a \emph{string diagram} we mean a Coxeter diagram with each connected component linear. A Coxeter group with a string diagram is called a \emph{string Coxeter group}.

More generally, a \emph{string group generated by involutions}, or \emph{sggi} for short, is defined as a group generated by pairwise distinct involutions $\rho_0,\ldots,\rho_{r-1}$ that satisfy the following property, called the \emph{commuting property}.
\[|i-j|>1\Rightarrow (\rho_i\rho_j)^2=1\]
Finally, a {\em string C-group} $\Gamma$ is an sggi satisfying the intersection property. In this case the underlying diagram for $\Gamma$ is a string diagram.
%$$ \xymatrix@-1pc{*+[o][F]{}  \ar@{-}[rr]_{p_1} && *+[o][F]{}  \ar@{-}[rr]_{p_2} && *+[o][F]{}  \ar@{-}[rr]_{p_3} && *+[o][F]{} \ar@{.}[rr] && *+[o][F]{}  \ar@{-}[rr]_{p_{r-2}} && *+[o][F]{} \ar@{-}[rr]_{p_{r-1}} &&*+[o][F]{} }$$
The \emph{(Schl\"afli) type} of $\Gamma$ is $\{p_1,\,\ldots,\,p_{r-1}\}$ where $p_i$ is the order of $\rho_{i-1}\rho_i,\,i\in\{1,\ldots,r-1\}$,
and the \emph{rank} of a string C-group is the size of its set of generators.
The \emph{dual} of a sggi is obtained by reversing the order of the generators; clearly the dual of a string C-group is itself a string C-group.

In~\cite{MS2002}, it was shown that string C-groups and abstract regular polytopes are in one-to-one correspondence.

Let $\Gamma=\langle \rho_0,\rho_1,\ldots,\rho_{r-1}\rangle$ be a sggi.
We denote by $\Gamma_j$ for $0\leq j\leq r-1$ the group generated by $\{\rho_i\,\,|\, i\neq j\}$,  $\Gamma_{<j}$ the group generated by $\{\rho_i\,\,|\, i< j\}$, $\Gamma_{>j}$ the group generated by $\{\rho_i\,\,|\, i>j\}$, and we write $\Gamma_{j_1,\ldots, j_k}$ with $\{j_1,\ldots, j_k\}\subset \{0,\ldots,r-1\}$, to denote the group generated by $\{\rho_i\,\,|\, i\neq j_1,\ldots, j_k\}$. The subgroups $\Gamma_i$ with $0\leq i\leq r-1$ are called the {\em maximal parabolic subgroups}.

Let $\Gamma=\langle\rho_0,\,\ldots,\,\rho_{r-1}\rangle$ be a sggi acting as a permutation group on a set $\{1,\,\ldots,\,n\}$.
We define the {\it permutation representation graph} $\mathcal{G}$ as the $r$-edge-labeled multigraph with $n$ vertices and with a single $i$-edge $\{a,\,b\}$ whenever $a\rho_i=b$ with $a\neq b$.  Note that each of the generators is an involution, thus the edges in our graph are not directed.  When $\Gamma$ is a string C-group, the multi-graph $\mathcal{G}$ is called a \emph{CPR graph}, as defined in \cite{pcpr}.

Let $\Gamma = \langle \rho_0,\ldots,\rho_{r-1} \rangle$, and let $\tau$ be an involution such that $\tau \not \in \Gamma$ and $\tau$ commutes with all of $\Gamma$.  We call the group $\Gamma^*= \langle \rho_i \tau^{\eta_i}\,|\, i\in \{0,\,\ldots,\,r-1\} \rangle$ where $\eta_i = 1$ if $i=k$ and 0 otherwise, the {\em sesqui-extension} of $\Gamma$ with respect to $\rho_k$ and $\tau$ (see \cite{flm}).

%%%%%%%%%%%%%%%%%%%%%%%%%%%%%%%%%%%%%%%%
%%%%%%%%%%%%%%%%%%%%%%%%%%%%%%%%%%%%%%%%
%%%%%%%%%%%%%%%%%%%%%%%%%%%%%%%%%%%%%%%%
\section{Intransitive maximal parabolic subgroups}\label{sectionintrans}
We start this section with the definition of fracture graphs. These graphs will play a central rule in the case where all maximal parabolic subgroups of $\Gamma$ are intransitive.
\begin{definition}\label{def1}
Let $\mathcal{G}$ be a permutation representation graph of $\Gamma$.
A \emph{fracture graph} $\mathcal{F}$ of $\Gamma$ is a subgraph of $\mathcal{G}$ containing all vertices of $\mathcal{G}$ and  one edge of each label, chosen in such a way that each $i$-edge joins two vertices $a_i$ and $b_i$ that are in distinct $\Gamma_i$-orbits.
\end{definition}

Observe that fracture graphs have exactly $r$ edges and are well defined as long as every $\Gamma_i$ is intransitive.

Throughout the remainder of this section, $\Gamma$  denotes a transitive permutation group of degree $n$ that is a string C-group of rank $r$ with all maximal parabolic subgroups of $\Gamma$ intransitive; $\mathcal{G}$ is the CPR-graph of $\Gamma$ and $\mathcal{F}$ is a fracture graph of $\Gamma$.

In the following figures, we represent the edges of  $\mathcal{G}$ that are not in $\mathcal{F}$ by dashed edges.
We use dotted lines  between two edges with respective labels $i$ and $j$ to represent edges with consecutive labels $i+1, \ldots, j-1$ that belong to $\mathcal{F}$.

\begin{lemma}\label{fracture}
\begin{enumerate}
\item $\mathcal{F}$  has no cycles;
\item $\mathcal{F}$ has $c$ connected components if and only if $r = n - c$;
\item If there exist two edges $\{a,b\}$ with distinct labels $i$ and $j$ in $\mathcal{G}$, then $a$ and $b$ are in distinct connected components of $\mathcal{F}$;
\item If there exist two $i$-edges $\{a, b\}$ and $\{c, d\}$ in $\mathcal{G}$, then not all vertices $a,b,c,d$ are in a same connected component of $\mathcal{F}$.
\end{enumerate}
\end{lemma}

\begin{proof}
\begin{enumerate}
\item Let $\{a, b\}$ be the $i$-edge of $\mathcal{F}$. Suppose that this edge is in a cycle of $\mathcal{F}$. Then $a$ and $b$ are in the same $\Gamma_i$-orbit, a contradiction with Definition~\ref{def1}.
\item This is a consequence of (1), which shows that $\mathcal{F}$ is a forest.
\item Suppose on the contrary that $a$ and $b$ are in a same connected component of $\mathcal{F}$. Then there exists a path in $\mathcal{F}$ from $a$ to $b$. Let $\{c,d\}$ be an $l$-edge of this path. The vertices $c$ and $d$ are then necessarily in the same $\Gamma_l$-orbit, a contradiction with Definition~\ref{def1}.
\item
At least one of the $i$-edges, say $\{a,b\}$, of $\mathcal{G}$ is not in $\mathcal{F}$.
Suppose that $a, b, c, d$ are all in the same connected component of $\mathcal{F}$. Then there is a path in $\mathcal{F}$ from $a$ to $b$. If $\{e,f\}$ is an $l$-edge on this path, then the vertices $e$ and $f$ are in the same $\Gamma_l$-orbit, a contradiction as before. Thus not all vertices $a,b,c,d$ are in the same connected component of $\mathcal{F}$.
\end{enumerate}
\end{proof}

\begin{lemma}\label{ic}
If a cycle $C$ of $\mathcal{G}$ contains the $i$-edge of $\mathcal{F}$, then $C$ contains another $i$-edge.
\end{lemma}
\begin{proof}
Suppose that a cycle of $\mathcal{G}$ contains exactly one $i$-edge $e=\{a,b\}$, then $a$ and $b$ are in the same $\Gamma_i$-orbit and therefore $e$ is not in $\mathcal{F}$ by Definition \ref{def1}.
\end{proof}

%..........................................
\begin{lemma}\label{swap}
If there is a cycle $C$ in $\mathcal{G}$  containing exactly two $i$-edges $e_1, e_2$, such that $e_1$ is in $\mathcal{F}$, then there is another  fracture graph $\mathcal{F}'$ obtained by removing $e_1$ and adding $e_2$.
\end{lemma}
\begin{proof}
Suppose that $e_2=\{a,b\}$ with $a$ and $b$ in the same $\Gamma_i$-orbit. Then $e_1$ is the unique $i$-edge in a cycle of $\mathcal{G}$, contradicting Lemma~\ref{ic}. Thus replacing $e_1$ by $e_2$ we obtain another fracture graph.
\end{proof}
%..........................................

\begin{lemma}\label{key}
If $a$ and $b$ are vertices in the same connected component of $\mathcal{F}$, and  $e=\{a,b\}$ is an $i$-edge in $\mathcal{G}$, then $e$ is in $\mathcal{F}$.
\end{lemma}
\begin{proof}
As $a$ and $b$ are vertices in the same connected component of $\mathcal{F}$, there exists a path $P$ in  $\mathcal{F}$ connecting $a$ to $b$. Suppose that $e$ is not in $\mathcal{F}$. Then there exists a $j$-edge in $P$ with $j\neq i$. As $e$ is an $i$-edge of $\mathcal{G}$ there exists a cycle containing the $j$-edge of $\mathcal{F}$ that does not contain any other $j$-edge, contradicting Lemma~\ref{ic}.
\end{proof}
%..........................................

\begin{lemma}\label{square}
Let $v$ and $w$ be vertices of an alternating square as in the following figure.
$$ 
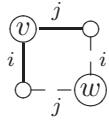
\xymatrix@-1pc{   *+[o][F]{v} \ar@{-}[d]_i \ar@{-}[r]^j  & *+[o][F]{}\ar@{--}[d]^i \\
*+[o][F]{} \ar@{--}[r]_j& *+[o][F]{w}}$$
Then $v$ and $w$ are in different connected components of $\mathcal{F}$.
\end{lemma}
\begin{proof}
Suppose that $v$ and $w$ are in the same connected component of $\mathcal{F}$. Then also $v'$, as in the following picture, is in the same connected component of $w$. Hence the $i$-edge $\{v',w\}$ must be in $\mathcal{F}$, by Lemma~\ref{key}. Thus $\mathcal{F}$ has two $i$-edges, a contradiction.
$$ 
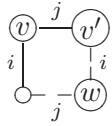
\xymatrix@-1pc{   *+[o][F]{v} \ar@{-}[d]_i \ar@{-}[r]^j  & *+[o][F]{v'}\ar@{--}[d]^i \\
*+[o][F]{} \ar@{--}[r]_j& *+[o][F]{w}}$$
\end{proof}

%..........................................
\begin{lemma}\label{square2}
Consider the alternating square as in the following figure, having an $i$-edge in $\mathcal{F}$ and both $j$-edges not in $\mathcal{F}$.
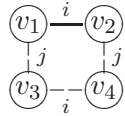
$$ \xymatrix@-1pc{   *+[o][F]{v_1} \ar@{-}[r]^i \ar@{--}[d]^j  & *+[o][F]{v_2}\ar@{--}[d]^j \\
*+[o][F]{v_3} \ar@{--}[r]_i& *+[o][F]{v_4}}$$
Then $v_1$, $v_3$ and $v_4$ are in different connected components of $\mathcal{F}$.
\end{lemma}
\begin{proof}
As $\{v_1,v_3\}$,  $\{v_3,v_4\}$ and $\{v_2,v_4\}$ are not edges of $\mathcal{F}$, by Lemma~\ref{key} each of these pairs of vertices are in different components of $\mathcal{F}$. Moreover, all three vertices $v_1$, $v_3$ and $v_4$ must belong to different connected components of $\mathcal{F}$, otherwise there is a cycle in $\mathcal{G}$ containing only one $i$-edge of  $\mathcal{F}$ contradicting Lemma~\ref{ic}.
\end{proof}
From now on, we assume that $r \geq n-4$.
\begin{lemma}\label{dg<=3}
Let $1\leq l \leq 4$, and $n\geq 3+2l$ when $r = n-l$. 
The degree of a vertex of $\mathcal{F}$ is at most $3$. Moreover, a vertex of degree $3$ in $\mathcal{F}$ has degree $3$ also in $\mathcal{G}$.
\end{lemma}
\begin{proof}
Suppose that there exists a vertex $v$ of degree $\geq 4$ in $\mathcal{F}$. Let $v$ be incident to 4 edges with labels $i, j, k, l$. Without loss of generality, suppose that $i<j<k<l$. Then $v$ is a vertex on three alternating squares as in the following figure.

$$ \xymatrix@-1pc{                         & *+[o][F]{}   \ar@{--}[r]^l      \ar@{-}[d]^j      & *+[o][F]{c} \ar@{--}[d]^j\\
 *+[o][F]{}    \ar@{--}[d]^i \ar@{-}[r]^k        & *+[o][F]{v}  \ar@{-}[d]^i\ar@{-}[r]_l & *+[o][F]{}  \ar@{--}[d]^i\\
 *+[o][F]{a}   \ar@{--}[r]_k & *+[o][F]{}  \ar@{--}[r]_l & *+[o][F]{b} }$$

By Lemmas~\ref{square} and \ref{ic}, $v,a,b,c$ are in different connected components of  $\mathcal{F}$. Therefore $\mathcal{F}$ has at least 4 connected components. By Lemma~\ref{fracture}~(2), $r=n-4$ and $\mathcal{F}$ has exactly 4 connected components. 

Suppose that $n\geq 9$ there exists a vertex $d$ incident, in $\mathcal{F}$, to one vertex of the previous picture. We may assume without loss of generality that $d$ is incident to $c$ in $\mathcal{F}$. Similar arguments as the ones below permit to get contradictions for the other cases. As $j<k<l$ the label $s$ of the edge $\{c,d\}$ cannot be consecutive with both $j$ and $l$. Suppose that $s$ is not consecutive with $j$. Then there is an alternating square containing $c$ and a $j$-edge as in the following picture.
$$ \xymatrix@-1pc{             & *+[o][F]{}   \ar@{--}[r]^l      \ar@{-}[d]^j      & *+[o][F]{c} \ar@{--}[d]^j\ar@{-}[r]^s& *+[o][F]{d}\ar@{--}[d]^j\\
 *+[o][F]{}    \ar@{--}[d]^i \ar@{-}[r]^k        & *+[o][F]{v}  \ar@{-}[d]^i\ar@{-}[r]_l & *+[o][F]{}  \ar@{--}[d]^i\ar@{--}[r]_s & *+[o][F]{e} \\
 *+[o][F]{a}   \ar@{--}[r]_k & *+[o][F]{}  \ar@{--}[r]_l & *+[o][F]{b}& } $$
Again by Lemmas~\ref{swap}, \ref{square} and \ref{ic}, the vertices $v, a,b,c$ and $e$ are in different connected components of $\mathcal{F}$, a contradiction with the fact that there are exactly four connected components. If $s$ is consecutive with $j$, then it is not consecutive with $l$ and we get the same contradiction using $l$ instead of $j$.

 Finally suppose that a vertex $v$ has degree 3 in $\mathcal{F}$ and has a higher degree in  $\mathcal{G}$. Let $v$ be incident to 4 edges with labels $i, j, k, l$. Without loss of generality, suppose that $i<j<k<l$. Suppose that the $j$-edge containing $v$ is not in $\mathcal{F}$. 
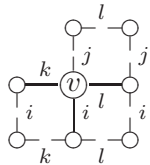
$$ \xymatrix@-1pc{                         & *+[o][F]{}   \ar@{--}[r]^l      \ar@{--}[d]^j      & *+[o][F]{} \ar@{--}[d]^j\\
 *+[o][F]{}    \ar@{--}[d]^i \ar@{-}[r]^k        & *+[o][F]{v}  \ar@{-}[d]^i\ar@{-}[r]_l & *+[o][F]{}  \ar@{--}[d]^i\\
 *+[o][F]{}   \ar@{--}[r]_k & *+[o][F]{}  \ar@{--}[r]_l & *+[o][F]{} }$$
By Lemma~\ref{swap} we may assume that both $j$-edges of the figure above are not in $\mathcal{F}$. Then by Lemma~\ref{square2} and \ref{ic}, we get 5 connected components for $\mathcal{F}$, a contradiction with Lemma~\ref{fracture}(2). If we assume that another edge incident to $v$ is not in $\mathcal{F}$, we get the same contradiction.
\end{proof}
%..........................................
\begin{lemma}\label{dg=3}
Let $1\leq l \leq 4$, and $n\geq 3+2l$ when $r = n-l$. 
If $\Gamma$ has a fracture graph having a vertex of degree $3$ then $r\in\{n-4, n-3\}$ and $\Gamma$ admits, up to duality, a fracture graph $\mathcal{F}$ from the following list.
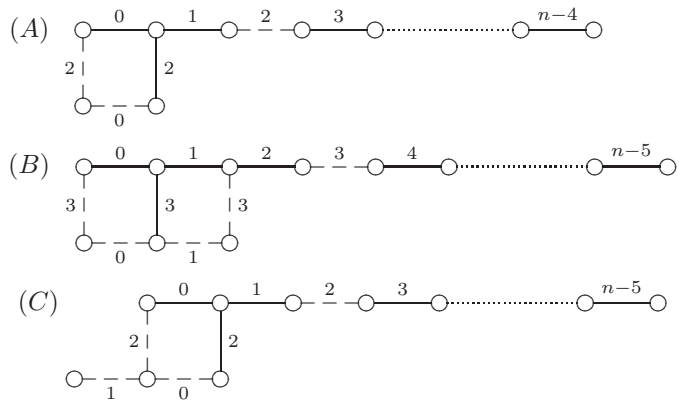
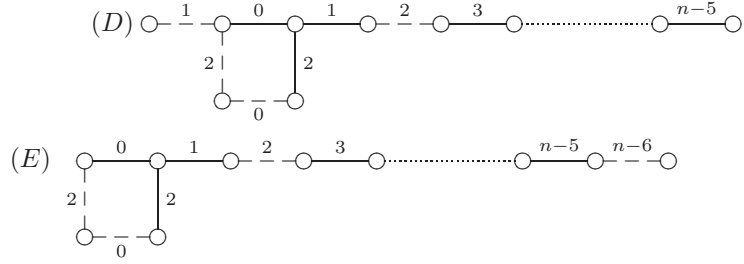
\begin{small}
\begin{center}
$(A)\; \xymatrix@-.2pc{
     *+[o][F]{}   \ar@{--}[d]_2  \ar@{-}[r]^0     & *+[o][F]{}  \ar@{-}[d]^2  \ar@{-}[r]^1  &   *+[o][F]{}  \ar@{--}[r]^2 &  *+[o][F]{} \ar@{-}[r]^3  &*+[o][F]{}\ar@{.}[rr]&& *+[o][F]{}\ar@{-}[r]^{n-4}&*+[o][F]{}\\
     *+[o][F]{}  \ar@{--}[r]_0     & *+[o][F]{}    &   &&&&&&}$

$(B)\; \xymatrix@-.2pc{
         *+[o][F]{}   \ar@{--}[d]_3  \ar@{-}[r]^0     & *+[o][F]{}  \ar@{-}[d]^3  \ar@{-}[r]^1  & *+[o][F]{} \ar@{--}[d]^3 \ar@{-}[r]^2 & *+[o][F]{} \ar@{--}[r]^3 & *+[o][F]{} \ar@{-}[r]^4& *+[o][F]{} \ar@{.}[rr] &&*+[o][F]{} \ar@{-}[r]^{n-5}&*+[o][F]{}\\
      *+[o][F]{}  \ar@{--}[r]_0     & *+[o][F]{}  \ar@{--}[r]_1   &  *+[o][F]{}&  &&&&&}$
$(C)\; \xymatrix@-.2pc{
     &*+[o][F]{}   \ar@{--}[d]_2  \ar@{-}[r]^0     & *+[o][F]{}  \ar@{-}[d]^{2}  \ar@{-}[r]^1  & *+[o][F]{}\ar@{--}[r]^2 & *+[o][F]{}\ar@{-}[r]^3 & *+[o][F]{}\ar@{.}[rr]&& *+[o][F]{}\ar@{-}[r]^{n-5}&*+[o][F]{}\\
    *+[o][F]{}  \ar@{--}[r]_1     & *+[o][F]{}  \ar@{--}[r]_0     & *+[o][F]{}    &   &&&&&}$

$(D) \; \xymatrix@-.2pc{
    *+[o][F]{}  \ar@{--}[r]^1 &  *+[o][F]{}   \ar@{--}[d]_2  \ar@{-}[r]^0     & *+[o][F]{}  \ar@{-}[d]^2  \ar@{-}[r]^1  & *+[o][F]{}\ar@{--}[r]^2 & *+[o][F]{}\ar@{-}[r]^3 & *+[o][F]{}\ar@{.}[rr]&& *+[o][F]{}\ar@{-}[r]^{n-5}&*+[o][F]{}\\
      & *+[o][F]{}  \ar@{--}[r]_0     & *+[o][F]{}    &   &&&&&}$

$(E)\; \xymatrix@-.2pc{
      *+[o][F]{}   \ar@{--}[d]_2  \ar@{-}[r]^0     & *+[o][F]{}  \ar@{-}[d]^2  \ar@{-}[r]^1  & *+[o][F]{}\ar@{--}[r]^2 & *+[o][F]{}\ar@{-}[r]^3 & *+[o][F]{}\ar@{.}[rr]&& *+[o][F]{}\ar@{-}[r]^{n-5}&*+[o][F]{}\ar@{--}[r]^{n-6} &*+[o][F]{}\\
      *+[o][F]{}  \ar@{--}[r]_0     & *+[o][F]{}    &   &&&&&&&&}$

\end{center}
\end{small}
\end{lemma}
\begin{proof}
By Lemma~\ref{dg<=3}, we know that the degree of each vertex of $\mathcal{F}$ is at most 3.
Assume $v$ is a vertex of degree 3 in $\mathcal{F}$ (and hence also in $\mathcal{G}$).
Let $i,j,k$ be the labels of the edges incident to $v$. Without loss of generality we may assume that $i<j<k$ and we may consider, up to duality, the following cases: (1) $j\neq i+ 1$ and $k\neq j+1$; (2) $j= i+1$ and $k\neq i+2$; (3) $j=i+ 1$ and $k= i+2$.

Case (1): $j\neq i+ 1$ and $k\neq j+1$; there are three alternating squares in $\mathcal{G}$ as in the following figure.
$$ 
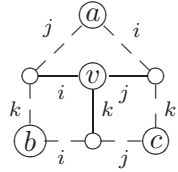
\xymatrix@-1pc{                         & *+[o][F]{a} & \\
        *+[o][F]{}  \ar@{--}[ru]^j  \ar@{--}[d]_k  \ar@{-}[r]_i     & *+[o][F]{v}  \ar@{-}[d]^k  \ar@{-}[r]_j  & *+[o][F]{} \ar@{--}[d]^k \ar@{--}[ul]_i \\      *+[o][F]{b}  \ar@{--}[r]_i     & *+[o][F]{}  \ar@{--}[r]_j   &  *+[o][F]{c} }$$

In this case $v,a,b$ and $c$ are, by Lemmas~\ref{ic} and \ref{square}, in different connected components. As $n\geq 11$, there is a vertex $d$ incident, in  $\mathcal{F}$, to one of the vertices of the previous picture. By Lemmas~\ref{dg<=3} and \ref{swap}, $d$ must be incident to either $a$, $b$ or $c$. 
Assume without loss of generality that $d$ is incident to $a$ and let $l$ be the label of the edge $\{a,d\}$. Then $l$ must be consecutive with both $i$ and $j$ and the degree of $d$ must be 1 in $\mathcal{G}$, for otherwise we get another square with a vertex in a fifth component of $\mathcal{F}$.  Hence, $n\leq 7+3=10$ and $r=n-4$, a contradiction.

Case (2): $j= i+1$ and $k\neq i+2$; there are two alternating squares as in the figure below:
$$ \xymatrix@-1pc{
        *+[o][F]{a}   \ar@{--}[d]_k  \ar@{-}[r]^i     & *+[o][F]{v}  \ar@{-}[d]^k  \ar@{-}[r]^{i+1}  & *+[o][F]{b} \ar@{--}[d]^k \\
             *+[o][F]{d}  \ar@{--}[r]_i     & *+[o][F]{c}  \ar@{--}[r]_{i+1}   &  *+[o][F]{e} }$$
By Lemmas~\ref{swap} and \ref{dg<=3}, $v$ and $c$ are both of degree 3 in $\mathcal{G}$.
As $n\geq 9$ either $a$, $b$, $d$ or $e$ is incident to another vertex. By Lemma~\ref{swap}, we may consider without loss of generality that $b$ is incident to a vertex $f$ (not necessarily in $\mathcal{F}$).
As $v$ has degree 3, the label of $\{b,f\}$ must be $i+2$. Now we consider separately the case $k\neq i+3$ and $k= i+3$.

Case (2.1): suppose that $k\neq i+3$. If $\{b,f\}$ is not in $\mathcal{F}$, by Lemmas~\ref{swap}, \ref{square2} and \ref{ic}, $\mathcal{F}$ has at least five connected components, a contradiction with Lemma~\ref{fracture} (2).
Hence $\{b,f\}$ is in $\mathcal{F}$.
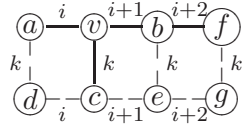
$$ \xymatrix@-1pc{
         *+[o][F]{a}   \ar@{--}[d]_k  \ar@{-}[r]^i     & *+[o][F]{v}  \ar@{-}[d]^k  \ar@{-}[r]^{i+1}  & *+[o][F]{b} \ar@{--}[d]^k \ar@{-}[r]^{i+2} & *+[o][F]{f}\ar@{--}[d]^k  \\
      *+[o][F]{d}  \ar@{--}[r]_i     & *+[o][F]{c}  \ar@{--}[r]_{i+1}   &  *+[o][F]{e}\ar@{--}[r]_{i+2} &  *+[o][F]{g}}$$

Now the vertices $v$, $d$, $e$ and $g$ are in different components of $\mathcal{F}$, and
the degree of $b$ and $e$ in $\mathcal{G}$ is 3, by Lemmas~\ref{swap} and \ref{dg<=3}.
As $n\geq 11$ either $a$, $b$, $f$ or $g$ is incident, in $\mathcal{F}$, to another vertex. Assume without loss of generality that $f$ is incident to another vertex $h$ in $\mathcal{F}$.  Now the label of $\{f,h\}$ must be consecutive with both $i+2$ and $k$ . But then the degree of $a$, $d$ and $g$ is two in $\mathcal{G}$, and the degree of $h$ is one in $\mathcal{G}$.  Therefore $n=9$, a contradiction.

Case (2.2): suppose that $k=i+3$. In this case we assume that the vertices $a$, $d$ and $e$ have degree two, otherwise we get back to Case (2.1).
If $\{b,f\}$ is not in $\mathcal{F}$, then $f$ must be incident, in  $\mathcal{F}$, to another vertex $h$, and the label of $\{f,h\}$ cannot be consecutive with $i+2$, giving a contradiction with Lemmas~\ref{swap} and \ref{dg<=3}. Thus $\{b,f\}$ is in $\mathcal{F}$. 
As $\mathcal{F}$ has at least 3 components, $n\geq 9$,  thus $f$ must be incident with another vertex $g$. The label of the edge $\{f,g\}$ must be consecutive with $i+2$, therefore the edge $\{f,g\}$ is not in $\mathcal{F}$. All vertices of the component of $\mathcal{F}$ containing the vertex $g$ have degree at most 2 and incident edges have consecutive labels. We get the possibility $(B)$ given in this lemma.

Case (3): $j=i+ 1$ and $k= i+2$; we have one alternating square as in the following figure.
$$ 
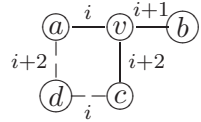
\xymatrix@-1pc{
     *+[o][F]{a}   \ar@{--}[d]_{i+2}  \ar@{-}[r]^i     & *+[o][F]{v}  \ar@{-}[d]^{i+2}  \ar@{-}[r]^{i+1}  & *+[o][F]{b} \\
     *+[o][F]{d}  \ar@{--}[r]_i     & *+[o][F]{c}    &   }$$

If the degree of $a$ is greater than 1 in $\mathcal{F}$, then by Lemma~\ref{swap} there exists a fracture graph with a vertex of degree 3 adjacent to edges with labels as in one of cases (1) or (2). Thus we assume that the degree of $a$ is one in $\mathcal{F}$.
By the same reasoning the degree of $c$ in $\mathcal{F}$ is one and  $d$ is an isolated vertex of $\mathcal{F}$.
The degree of $b$ is one in $\mathcal{F}$, as another edge of $\mathcal{F}$ incident with $b$ would have a label not consecutive with $i+1$, implying that $v$ has degree 4 and contradicting Lemma~\ref{dg<=3}.
Hence, as  $\mathcal{F}$ has at three components, $r\geq n-3$.

Now suppose that there is another vertex of degree 3 in $\mathcal{F}$. In that case, $\mathcal{F}$ has 4 components: two components with 4 vertices and two isolated vertices, hence $n=4+4+1+1=10$, a contradiction. Thus only $v$ has degree 3 in $\mathcal{F}$.

If $r=n-3$, we get only one possibility for $\mathcal{F}$ corresponding to graph (A).

If $r=n-4$,  $\mathcal{F}$  has a component containing the vertex $v$ of degree 3, a component that is a isolated vertex $d$ and another two components only have vertices of degree $\leq 2$.
If $a$ is incident to another vertex $e$ in $\mathcal{G}$  then the label of $\{a, e\}$ must be consecutive with $i$ by Lemma~\ref{dg<=3}.
Suppose it is $i-1$ then we have an alternating square with labels $i-1$ and $i+2$, as in the following figure.
$$ \xymatrix@-1pc{
    *+[o][F]{e}   \ar@{--}[d]_{i+2}  \ar@{--}[r]^{i-1}     &  *+[o][F]{a}   \ar@{--}[d]_{i+2}  \ar@{-}[r]^i     & *+[o][F]{v}  \ar@{-}[d]^{i+2}  \ar@{-}[r]^{i+1}  & *+[o][F]{b} \\
    *+[o][F]{f}  \ar@{--}[r]_{i-1}     & *+[o][F]{d}  \ar@{--}[r]_i     & *+[o][F]{c}    &   }$$
By Lemmas~\ref{square2} and \ref{ic}, $v$, $d$, $e$ and $f$ are in different components of $\mathcal{F}$.
As $n>7$ one of the vertices of the figure above is incident (in $\mathcal{F}$) to another vertex and we get a fifth connected component of $\mathcal{F}$, a contradiction.
Hence the label of $\{a,e\}$ is $i+1$. Similarly if $c$ or $d$ is incident to another vertex, the label of that edge must be $i+1$.
If two vertices of the set $\{a,c, d\}$ have degree 3 in $\mathcal{G}$, then we get two additional components that are isolated vertices. In that case, $n\leq 7$, a contradiction.
Therefore, only one vertex of $\{a,c, d\}$ can have degree 3 in $\mathcal{G}$.
If $d$ has degree 3 then the vertex $e$ incident to $d$ must be an isolated vertex of $\mathcal{F}$. Moreover, either $e$ or $b$ has degree greater than 1 in $\mathcal{G}$ (but not both). Thus we get two possibilities for $\mathcal{F}$, one corresponding to graph $(C)$ of this lemma and the other is the following graph.
$$(C')\; 
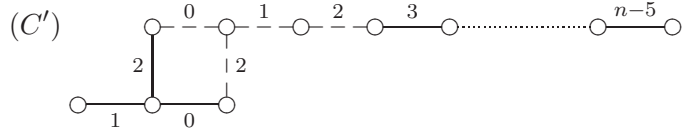
\xymatrix@-.2pc{
     &*+[o][F]{}   \ar@{-}[d]_2  \ar@{--}[r]^0     & *+[o][F]{}  \ar@{--}[d]^{2}  \ar@{--}[r]^1  & *+[o][F]{}\ar@{--}[r]^2 & *+[o][F]{}\ar@{-}[r]^3 & *+[o][F]{}\ar@{.}[rr]&& *+[o][F]{}\ar@{-}[r]^{n-5}&*+[o][F]{}\\
    *+[o][F]{}  \ar@{-}[r]_1     & *+[o][F]{}  \ar@{-}[r]_0     & *+[o][F]{}    &   &&&&&}$$
In $(C')$ both $0$-edges, $1$-edges and $2$-edges are between vertices in different  $\Gamma_i$-orbits ($i=0,1$, or $2$ respectively). Therefore if $\Gamma$ admits the fracture graph $C$, it also admits the fracture graph $C'$ and vice-versa, hence in this lemma only the possibility $(C)$ need to be listed. By the same reasoning, if $a$ or $c$ has degree 3, $\Gamma$ has a fracture graph $(D)$.

Now suppose that $a, c$ and $d$ have degree 2.
Then $b$ is incident to a vertex $e$ in $\mathcal{G}$ and the label of $\{b,e\}$ must be consecutive with $i+1$. Without loss of generality, we may assume it is $i+2$.
The label of an edge incident with  $e$ must be consecutive with $e$, thus $e$ is either an isolated vertex of $\mathcal{F}$ or has degree one in $\mathcal{F}$.
Suppose that $e$ is an isolated vertex of $\mathcal{F}$.
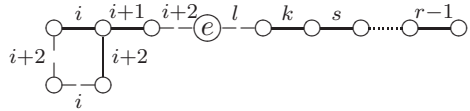
$$ \xymatrix@-1pc{
      *+[o][F]{}   \ar@{--}[d]_{i+2}  \ar@{-}[r]^i     & *+[o][F]{}  \ar@{-}[d]^{i+2}  \ar@{-}[r]^{i+1}  & *+[o][F]{}\ar@{--}[r]^{i+2} & *+[o][F]{e}\ar@{--}[r]^l & *+[o][F]{}\ar@{-}[r]^k& *+[o][F]{}\ar@{-}[r]^s& *+[o][F]{}\ar@{.}[r] &*+[o][F]{}\ar@{-}[r]^{r-1}&*+[o][F]{}\\
      *+[o][F]{}  \ar@{--}[r]_i     & *+[o][F]{}    &   &&&&&&}$$
If $l=i+1$ then $k\in\{i,i+2\}$, which is not possible. Hence $l=i+3$, $k=i+4$, $s=i+3$, but then $n=9$, a contradiction.
Thus $e$ is not isolated.
Suppose that the last component is an isolated vertex, then there is one possibility for $\mathcal{F}$, corresponding to graph (E).

Now suppose that $\mathcal{F}$ has only one isolated vertex (the one in the square). 
$$ 
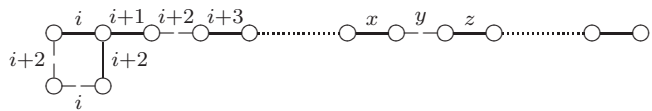
\xymatrix@-1pc{
      *+[o][F]{}   \ar@{--}[d]_{i+2}  \ar@{-}[r]^i     & *+[o][F]{}  \ar@{-}[d]^{i+2}  \ar@{-}[r]^{i+1}  & *+[o][F]{}\ar@{--}[r]^{i+2} & *+[o][F]{}\ar@{-}[r]^{i+3} &  *+[o][F]{}\ar@{.}[rr]& &*+[o][F]{}\ar@{-}[r]^{x} &*+[o][F]{}\ar@{--}[r]^y&*+[o][F]{}\ar@{-}[r]^z &  *+[o][F]{}\ar@{.}[rr]&& *+[o][F]{}\ar@{-}[r]&*+[o][F]{}\\
      *+[o][F]{}  \ar@{--}[r]_i     & *+[o][F]{}    &   & & &&& &&&&&}$$
In that case, $y\in\{x-1,x+1\}\cap\{z-1,z+1\}$ and $x<z$ imply that $y=x+1=z-1$. Hence $z=x+2$ and the unique possibility for $\mathcal{F}$ is graph $(E')$.
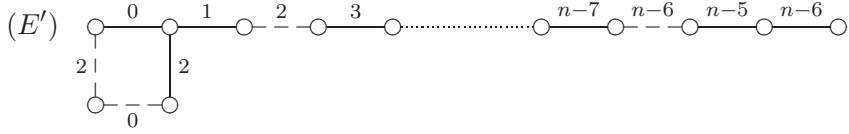
$$(E')\; \xymatrix@-.2pc{
      *+[o][F]{}   \ar@{--}[d]_2  \ar@{-}[r]^0     & *+[o][F]{}  \ar@{-}[d]^2  \ar@{-}[r]^1  & *+[o][F]{}\ar@{--}[r]^2 & *+[o][F]{}\ar@{-}[r]^3 & *+[o][F]{}\ar@{.}[rr]&& *+[o][F]{}\ar@{-}[r]^{n-7}& *+[o][F]{}\ar@{--}[r]^{n-6} &*+[o][F]{}\ar@{-}[r]^{n-5}&*+[o][F]{}\ar@{-}[r]^{n-6} &*+[o][F]{}\\
      *+[o][F]{}  \ar@{--}[r]_0     & *+[o][F]{}    &   &&&&&&&&&&}$$
As the dashed edges are uniquely determined, if $(E)$ is a fracture graph of $\Gamma$ then $(E')$ is also a fracture graph of $\Gamma$ and vice-versa, so only one of these two graphs needs to be listed in this lemma.
\end{proof}

%..........................................

\begin{lemma}\label{T}
Let $1\leq l \leq 4$, and $n\geq 3+2l$ when $r = n-l$. 
If a fracture graph $\mathcal{F}$ of $\Gamma$ has adjacent edges with non-consecutive labels, then $\Gamma$ also has a fracture graph with a vertex of degree $3$.
\end{lemma}
\begin{proof}
Suppose that $\mathcal{F}$ has two adjacent edges with non-consecutive labels $i$ and $j$
and assume, by way of contradiction, that $\Gamma$ has no fracture graph with a vertex of degree 3.
The graph $\mathcal{G}$ has an alternating square with labels $i$ and $j$.
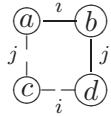
$$ \xymatrix@-1pc{
     *+[o][F]{a}   \ar@{-}[r]^i  \ar@{--}[d]_j     & *+[o][F]{b}  \ar@{-}[d]^j  \\
    *+[o][F]{c}  \ar@{--}[r]_i  & *+[o][F]{d}                                        }$$
There exist no vertex incident, in $\mathcal{F}$, to one of the vertices of the alternating $i,j$ square for, otherwise, using Lemma~\ref{swap}, we can create a fracture graph for $\Gamma$ with a vertex of degree 3.
Thus $\mathcal{F}$ has at least three components and  $n\geq 9$. Hence there is a vertex $e$ adjacent to one of the vertices of the square in $\mathcal{G}$. By Lemma~\ref{swap} we may assume that there exists a $k$-edge $\{a,e\}$ not in $\mathcal{F}$.
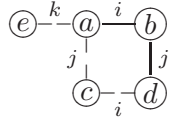
$$ \xymatrix@-1pc{
    *+[o][F]{e}  \ar@{--}[r]^k  & *+[o][F]{a}   \ar@{-}[r]^i  \ar@{--}[d]_j     & *+[o][F]{b}  \ar@{-}[d]^j  \\
   & *+[o][F]{c}  \ar@{--}[r]_i  & *+[o][F]{d}                                        }$$
If $k$ is not consecutive with both $i$ and  $j$, we have another two alternating squares and,  by Lemmas~\ref{square2} and \ref{ic}, we get five connected components, a contradiction with Lemma~\ref{fracture}(2). Thus $k$ must be consecutive either with $i$ or with $j$.
Suppose that $k=i-1$ and $j$ is not consecutive with $k$.
Then, by Lemmas~\ref{swap} and \ref{square2}, $\mathcal{F}$ has 4 connected components.
Therefore, $b$ and $d$ have degree 2 in $\mathcal{G}$.
Now let $e$ be incident to another vertex $f$ and let $l$ be the label of $\{e,f\}$.
If $l$ is not consecutive either with $i-1$ or with $j$, using Lemmas \ref{square2} and \ref{ic}, we get a fifth connected component, a contradiction with Lemma~\ref{fracture}(2). Thus $j=i-3$ and $l=i-2$. Moreover the vertex $g$ in the following figure must be isolated.
$$ 
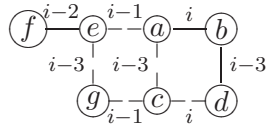
\xymatrix@-1pc{
   *+[o][F]{f}  \ar@{-}[r]^{i-2} &   *+[o][F]{e}\ar@{--}[d]_{i-3}    \ar@{--}[r]^{i-1} & *+[o][F]{a}   \ar@{-}[r]^i  \ar@{--}[d]_{i-3}     & *+[o][F]{b}  \ar@{-}[d]^{i-3}  \\
   &*+[o][F]{g}  \ar@{--}[r]_{i-1}  & *+[o][F]{c}  \ar@{--}[r]_i  & *+[o][F]{d}                                        }$$
As $n\geq 11$, $f$ must be $(i-1)$-incident to another vertex $h$, but then $h$ cannot be incident to any other vertex; thus we get $n\leq 8$, a contradiction.
Consequently the label $k$ of the edge $\{a,e\}$ is consecutive with both $i$ and $j$. 
Let $k=i+1$ and $j=i+2$.
Now $a$, $c$ and $e$ have degree 2 in $\mathcal{G}$ for, if one of these vertices has degree 3 in $\mathcal{G}$, the third edge must have label $i-1$ and in that case, we cannot extend the graph without creating too many connected components for $\mathcal{F}$. As $n\geq 11$, $e$ is $l$-adjacent to another vertex $f$ in $\mathcal{G}$.
Moreover $l$ must be consecutive with $i+1$, thus $\{e,f\}$ is not in $\mathcal{F}$ and $\mathcal{F}$ has 4 components.
We may assume that $l=i+2$ and then the other edge incident  to $f$ has label $i+1$. This gives $n<8$, a contradiction.
\end{proof}
%..........................................
\begin{lemma}\label{G1}
Let $1\leq l \leq 4$, and $n\geq 3+2l$ when $r = n-l$. 
Suppose that $\mathcal{G}$ has adjacent edges with nonconsecutive labels but there is no fracture graph having adjacent edges with nonconsecutive labels. 
Then $r\in\{n-4,n-3\}$ and $\mathcal{G}$ admits, up to a duality,  a fracture graph $\mathcal{F}$ of the following list.
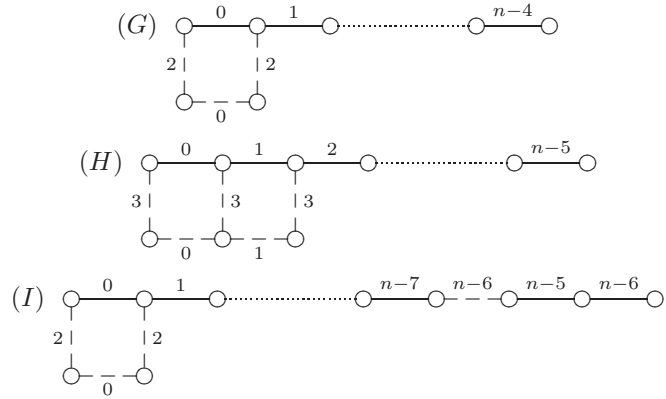
\begin{small}
\begin{center}
$ (G) \xymatrix@-.2pc{
     *+[o][F]{}   \ar@{-}[r]^0  \ar@{--}[d]_2     & *+[o][F]{}  \ar@{-}[r]^1 \ar@{--}[d]^2   & *+[o][F]{}  \ar@{..}[rr] && *+[o][F]{}  \ar@{-}[r]^{n-4}& *+[o][F]{}  \\
     *+[o][F]{}  \ar@{--}[r]_0     & *+[o][F]{} &    & & &  }$

$ (H) \xymatrix@-.2pc{
     *+[o][F]{}   \ar@{-}[r]^0  \ar@{--}[d]_3    & *+[o][F]{}  \ar@{-}[r]^1 \ar@{--}[d]^3 & *+[o][F]{} \ar@{--}[d]^3  \ar@{-}[r]^2 & *+[o][F]{} \ar@{.}[rr]&&*+[o][F]{} \ar@{-}[r]^{n-5} &*+[o][F]{}\\
     *+[o][F]{}  \ar@{--}[r]_0    & *+[o][F]{}   \ar@{--}[r]_1 &   *+[o][F]{} & &&&}$

$ (I) \xymatrix@-.2pc{
     *+[o][F]{}   \ar@{-}[r]^0  \ar@{--}[d]_2     & *+[o][F]{}  \ar@{-}[r]^1 \ar@{--}[d]^2 &*+[o][F]{}  \ar@{.}[rr] &&*+[o][F]{}  \ar@{-}[r]^{n-7} &*+[o][F]{}  \ar@{--}[r]^{n-6} &*+[o][F]{}  \ar@{-}[r]^{n-5} &*+[o][F]{}  \ar@{-}[r]^{n-6} &*+[o][F]{}\\
     *+[o][F]{}  \ar@{--}[r]_0     & *+[o][F]{}    &    }$

\end{center}
\end{small}
\end{lemma}
\begin{proof}
Let $i$ and $j$ be the labels of two adjacent edges in $\mathcal{G}$ with $|i-j|>2$.
There is an alternating square in $\mathcal{G}$ containing these edges. Let $\mathcal{F}$ be a fracture graph of $\mathcal{G}$.

Suppose that the vertices of the alternating square are in different connected components of $\mathcal{F}$ (which can happen only when $r=n-4$).
Consider a vertex $k$-incident to one of the vertices of the square.
As $\mathcal{F}$ has 4 connected components, therefore this $k$-edge is an edge of $\mathcal{F}$.
Suppose that $k$ and $j$ are not consecutive.
$$ \xymatrix@-1pc{
   *+[o][F]{}   \ar@{--}[r]^i \ar@{--}[d]^j&  *+[o][F]{}  \ar@{--}[d]^j \ar@{-}[r]^k  & *+[o][F]{}  \ar@{--}[d]^j  \\
   *+[o][F]{}  \ar@{--}[r]_i  & *+[o][F]{} \ar@{--}[r]_k   & *+[o][F]{}                                   }$$
Then by Lemmas~\ref{square2} and \ref{ic} we get five components of $\mathcal{F}$, a contradiction.
Thus $k$ must be consecutive with both $i$ and $j$. Therefore 3 components of $\mathcal{F}$ are isolated vertices and the component with more than one vertex must contain a  $(k-1)$-edge and a $(k+1)$-edge simultaneously, which is not possible as all incident edges of $\mathcal{F}$ have consecutive labels.

Now suppose that a pair of vertices of the alternating square are in the same connected component of $\mathcal{F}$.
Then by Lemma~\ref{key} one of the edges of the square is in $\mathcal{F}$.
Without loss of generality suppose that one of the $i$-edges is in $\mathcal{F}$.
Let $\{a,b\}$ be the $i$-edge of $\mathcal{F}$ and  $\{v,w\}$ be the other $i$-edge of the alternating square. By Lemma~\ref{square2}, $v$, $w$ and $a$  are in different connected components of $\mathcal{F}$.
As before there is an edge incident to one of the vertices of the square. By Lemma~\ref{swap} we may suppose without loss of generality that it connects $b$ to some other vertex $c$.

Consider first the case $r=n-3$. As $\mathcal{F}$ has 3 connected components,
$c$ is in the same connected component as one of the vertices of the alternating square. Hence by Lemmas~\ref{ic} and \ref{key}, $\{b,c\}$ is the $k$-edge of $\mathcal{F}$.
$$ \xymatrix@-1pc{
    *+[o][F]{a}  \ar@{-}[r]^i \ar@{--}[d]_j   & *+[o][F]{b}   \ar@{-}[r]^k  \ar@{--}[d]^j     & *+[o][F]{c}  \\
   *+[o][F]{v}  \ar@{--}[r]_i  & *+[o][F]{w}   &                                     }$$
By hypothesis $k$ must be consecutive with $i$.
Suppose that $k$ and $j$ are not consecutive. Hence there is an alternating $\{k,j\}$-square and by Lemmas~\ref{square} and \ref{ic} we get a fourth component, which is not the case we are dealing with.
Thus $k$ is consecutive with both $i$ and $j$.
Therefore $a$ has degree one in $\mathcal{F}$ and $v$ and $w$ are isolated vertices of $\mathcal{F}$. Then we get the graph $(G)$ of this lemma.

Now let $r=n-4$.
First suppose that $b$ and $c$ are in different connected components of $\mathcal{F}$.
Let $k$ be the label of the $\{b,c\}$.
Suppose that $k$ is not consecutive with $i$. Then, by Lemma~\ref{square2}, we have a fifth connected component of $\mathcal{F}$, a contradiction.
If $k$ is not consecutive with $j$ we have two alternating squares as in the following figure.
$$ \xymatrix@-1pc{
   *+[o][F]{a}   \ar@{-}[r]^i \ar@{--}[d]^j&  *+[o][F]{b}  \ar@{--}[d]^j \ar@{--}[r]^k  & *+[o][F]{c}  \ar@{--}[d]^j  \\
   *+[o][F]{v}  \ar@{--}[r]_i  & *+[o][F]{w} \ar@{--}[r]_k   & *+[o][F]{}                                   }$$
By Lemmas ~\ref{square2} and \ref{ic} either a $k$-edge or a $j$-edge is in $\mathcal{F}$. Then by  Lemma~\ref{swap} we get $b$ and $c$ in the same connected component, or two non-consecutive incident edges of $\mathcal{F}$, a contradiction.
Hence $k$ is consecutive with both $i$ and $j$ and only the vertex $b$ has degree greater than 2 in $\mathcal{G}$. All remaining incident edges must be consecutive.
$$ 
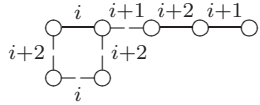
\xymatrix@-1pc{
   *+[o][F]{}   \ar@{-}[r]^i \ar@{--}[d]_{i+2}&  *+[o][F]{}  \ar@{--}[d]^{i+2} \ar@{--}[r]^{i+1}  & *+[o][F]{}\ar@{-}[r]^{i+2}  & *+[o][F]{}\ar@{-}[r]^{i+1} & *+[o][F]{} \\
   *+[o][F]{}  \ar@{--}[r]_i  & *+[o][F]{} &          &&                        }$$
It is clear that is not possible, thus we get no possibility for $\mathcal{G}$ when $n>7$.

Consider that $b$ and $c$ are in the same connected component of $\mathcal{F}$. Then by Lemma~\ref{key} the $k$-edge $\{b,c\}$ is in $\mathcal{F}$.
Moreover, by hypothesis, $k$ must be consecutive with $i$ so let $k=i+1$.

Suppose that $k$ is not consecutive with $j$, then we get two alternating squares as in the following figure.
$$ \xymatrix@-1pc{
     *+[o][F]{a}   \ar@{-}[r]^{i}  \ar@{--}[d]^{j}     & *+[o][F]{b}  \ar@{-}[r]^{i+1} \ar@{--}[d]^{j} & *+[o][F]{c} \ar@{--}[d]^{j} \\
     *+[o][F]{v}  \ar@{--}[r]_i     & *+[o][F]{w}   \ar@{--}[r]_{i+1} &   *+[o][F]{y} }$$
By Lemmas~\ref{square2} and \ref{ic} the vertices $a$, $v$, $w$ and $y$ are in different components.
There exists a vertex $d$ incident to one of the vertices of the previous picture, neither $b$ nor $w$, otherwise we get two non-consecutive incident edges of $\mathcal{F}$. By Lemma~\ref{swap}, we may assume $d$ is incident with $c$.
By hypothesis, the label of $\{c,d\}$ must now be $i+2$. Hence $j=i+3$ for, otherwise, we get a fifth component of $\mathcal{F}$. This gives fracture graph $(H)$.

Now suppose that $k$ is consecutive with both $i$ and $j$, and that all the vertices of the alternating square, except $b$, have degree two in $\mathcal{G}$. Then $\mathcal{F}$ is the graph $(I)$ or is the graph $(I')$ obtained from $(I)$ by interchanging the edge of label $n-6$ that is in $\mathcal{F}$ and the edge of label $n-6$ that is not in $\mathcal{F}$.
Only one of these two graphs needs to be listed in this lemma.
\end{proof}
%..........................................

The permutation graph  $\mathcal{G}$ of $\Gamma$ is \emph{linear} if adjacent edges of have consecutive labels. 
When $\mathcal{G}$ is linear, it is possible to give an ordering to the connected components of $\mathcal{F}$.  
A component of $\mathcal{F}$  is \emph{big} if it has at least four vertices.
We say that a fracture graph is \emph{maximal with respect to a component} if that component cannot become bigger in any other fracture graph.

In the following graphs we sometimes have more than one possibility for the label of an edge. If that is the case we write all possibilities for the label of that edge separated by a vertical bar. If one of the possibilities is a double edge we write a set with two labels instead of a single label.
\begin{proposition}\label{join}
Let $1\leq l \leq 4$, and $n\geq 3+2l$ when $r = n-l$. 
Suppose that $\mathcal{G}$ is linear and $\mathcal{F}$ is a maximal fracture graph of $\mathcal{G}$ with respect to the last big component. Consider two components of $\mathcal{F}$ at distance one in $\mathcal{G}$ as in the following figure.
$$ \ldots\xymatrix@-.3pc{
     *+[o][F]{}   \ar@{-}[r]^i    & *+[o][F]{}  \ar@{--}[r]^k  & *+[o][F]{}  \ar@{-}[r]^j& *+[o][F]{}}\;\ldots (i<j)$$
Then $k=i+1$, $j=i+2$ and  $\mathcal{F}$ has at least 3 components. Moreover, either the component containing the $i$-edge or the component containing the $j$-edge has exactly 3 vertices.
\end{proposition}
\begin{proof}
As $\mathcal{G}$ is linear, $k\in\{i-1,i+1\}\cap\{j-1,j+1\}$ with $i<j$. Hence $k=i+1$ and $j=i+2$.
As incident edges are consecutive, the edge with label $i+1$ must belong to a third component having exactly two vertices.

Let $\mathcal{C}$ be the component of $\mathcal{F}$ containing the $(i+2)$-edge. Suppose that the component consisting in a single  $(i+1)$-edge is after $\mathcal{C}$.

First suppose that  $\mathcal{C}$ has  more than 4 vertices.  This forces $\mathcal{F}$ to have at least five components, a contradiction. We show in the following figure what happens when $\mathcal{C}$ has exactly five vertices.
$$ \ldots\xymatrix@-.3pc{
     *+[o][F]{}   \ar@{-}[r]^{i-1}    &*+[o][F]{}   \ar@{-}[r]^i    & *+[o][F]{}  \ar@{--}[r]^{i+1}  & *+[o][F]{}  \ar@{-}[r]^{i+2}& *+[o][F]{}  \ar@{-}[r]^{i+3}&*+[o][F]{}  \ar@{-}[r]^{i+4}&*+[o][F]{}  \ar@{-}[r]^{i+5}&*+[o][F]{}  \ar@{--}[r]^{i+4}&*+[o][F]{}  \ar@{--}[r]^{i+3}&*+[o][F]{}  \ar@{--}[r]^{i+2}&*+[o][F]{}  \ar@{-}[r]^{i+1}&*+[o][F]{}}$$
If  $\mathcal{C}$ has exactly 4 vertices, then $\mathcal{F}$ has 4 components as in the following figure. 
$$ \ldots\xymatrix@-.3pc{
    *+[o][F]{}   \ar@{-}[r]^{i-2}    &*+[o][F]{}   \ar@{-}[r]^{i-1}    &*+[o][F]{}   \ar@{-}[r]^i    & *+[o][F]{}  \ar@{--}[r]^{i+1}  & *+[o][F]{}  \ar@{-}[r]^{i+2}& *+[o][F]{}  \ar@{-}[r]^{i+3}&*+[o][F]{}  \ar@{-}[r]^{i+4} &*+[o][F]{}  \ar@{--}[r]^{i+3}&*+[o][F]{}  \ar@{--}[r]^{i+2}&*+[o][F]{}  \ar@{-}[r]^{i+1}&*+[o][F]{}}$$
In this case $\mathcal{F}$ has two big components, as $n\geq 11$, and none is maximal, a contradiction.
Now suppose that $\mathcal{C}$ has exactly 2 vertices. Then we have one of the following fracture graphs with 4 components.
$$ \ldots
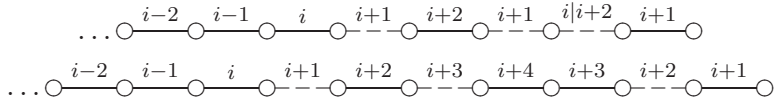
\xymatrix@-.3pc{
      *+[o][F]{}   \ar@{-}[r]^{i-2}    &*+[o][F]{}   \ar@{-}[r]^{i-1}    &*+[o][F]{}   \ar@{-}[r]^i    & *+[o][F]{}  \ar@{--}[r]^{i+1}  & *+[o][F]{}  \ar@{-}[r]^{i+2}& *+[o][F]{}  \ar@{--}[r]^{i+1}&*+[o][F]{}  \ar@{--}[r]^{i|i+2} &*+[o][F]{}  \ar@{-}[r]^{i+1}&*+[o][F]{}  }$$
   $$ \ldots\xymatrix@-.3pc{
     *+[o][F]{}   \ar@{-}[r]^{i-2}    &*+[o][F]{}   \ar@{-}[r]^{i-1}    &*+[o][F]{}   \ar@{-}[r]^i    & *+[o][F]{}  \ar@{--}[r]^{i+1}  & *+[o][F]{}  \ar@{-}[r]^{i+2}& *+[o][F]{}  \ar@{--}[r]^{i+3}&*+[o][F]{}  \ar@{-}[r]^{i+4}&*+[o][F]{}  \ar@{-}[r]^{i+3}  &*+[o][F]{}  \ar@{--}[r]^{i+2}&*+[o][F]{}  \ar@{-}[r]^{i+1}&*+[o][F]{}  }$$
In any case we have a contradiction with maximality of the unique big component of $\mathcal{F}$.

If the component being the  $(i+1)$-edge is before $\mathcal{C}$ we get the same contradiction with maximality of the last big component (indeed maximality fails in all big components of $\mathcal{F}$). 
\end{proof}
%----------------------------------
\begin{lemma}\label{linear}
Let $1\leq l \leq 4$, and $n\geq 3+2l$ when $r = n-l$. 
Suppose that $\mathcal{G}$ is linear, then $\mathcal{G}$ is, up to duality, one of the following graphs, where $0\leq i\leq r-4$. 

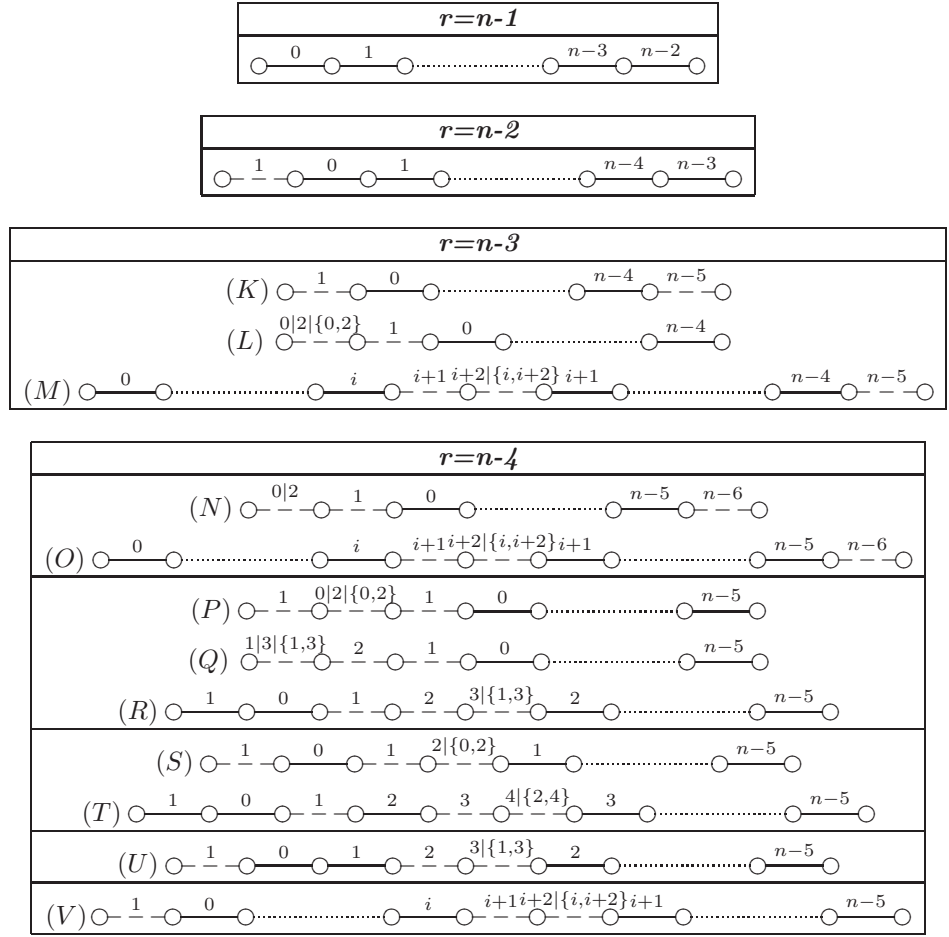
\begin{small}
$$\begin{tabular}{c}
\begin{tabular}{|c|}
\hline
$\textbf{r=n-1} $\\
\hline
$\xymatrix@-.2pc{ *+[o][F]{}   \ar@{-}[r]^0  & *+[o][F]{}  \ar@{-}[r]^1    & *+[o][F]{} \ar@{.}[rr] & & *+[o][F]{}  \ar@{-}[r]^{n-3} & *+[o][F]{}  \ar@{-}[r]^{n-2} & *+[o][F]{}}$\\
\hline
\end{tabular}\\
\\
\begin{tabular}{|c|}
\hline
$\textbf{r=n-2} $\\
\hline
$\xymatrix@-.2pc{ *+[o][F]{}   \ar@{--}[r]^1  & *+[o][F]{}  \ar@{-}[r]^0  & *+[o][F]{}  \ar@{-}[r]^1    & *+[o][F]{} \ar@{.}[rr] & & *+[o][F]{}  \ar@{-}[r]^{n-4} & *+[o][F]{}  \ar@{-}[r]^{n-3} & *+[o][F]{}}$\\
\hline
\end{tabular}\\
\\
\begin{tabular}{|c|}
\hline
$\textbf{r=n-3}$\\
\hline
$(K)\; \xymatrix@-.2pc{ *+[o][F]{}   \ar@{--}[r]^1  & *+[o][F]{}  \ar@{-}[r]^0    & *+[o][F]{} \ar@{.}[rr] & & *+[o][F]{}  \ar@{-}[r]^{n-4} & *+[o][F]{}  \ar@{--}[r]^{n-5} & *+[o][F]{}}$\\

$(L)\; \xymatrix@-.2pc{ *+[o][F]{}   \ar@{--}[r]^{0|2|\{0,2\}}  & *+[o][F]{}  \ar@{--}[r]^1    &  *+[o][F]{}   \ar@{-}[r]^0   & *+[o][F]{} \ar@{.}[rr] &&  *+[o][F]{}  \ar@{-}[r]^{n-4} & *+[o][F]{}}$\\

$ (M)\; \xymatrix@-.1pc{ *+[o][F]{}   \ar@{-}[r]^0    & *+[o][F]{}  \ar@{.}[rr]& & *+[o][F]{}  \ar@{-}[r]^{i}  & *+[o][F]{}  \ar@{--}[r]^{i+1} & *+[o][F]{}  \ar@{--}[r]^{i+2|\{i,i+2\}} & *+[o][F]{}  \ar@{-}[r]^{i+1}  & *+[o][F]{}  \ar@{.}[rr] && *+[o][F]{}  \ar@{-}[r]^{n-4}&*+[o][F]{}  \ar@{--}[r]^{n-5}& *+[o][F]{}}$\\
\hline
\end{tabular}\\
\\
\begin{tabular}{|c|}
\hline
$\textbf{r=n-4}$\\
\hline

$(N)\; \xymatrix@-.2pc{ *+[o][F]{}   \ar@{--}[r]^{0|2}  & *+[o][F]{} \ar@{--}[r]^1  &  *+[o][F]{} \ar@{-}[r]^0 & *+[o][F]{} \ar@{.}[rr] && *+[o][F]{}  \ar@{-}[r]^{n-5}& *+[o][F]{}  \ar@{--}[r]^{n-6}  & *+[o][F]{}}$\\

$(O)\;\xymatrix@-.2pc{ *+[o][F]{}   \ar@{-}[r]^0  &  *+[o][F]{} \ar@{.}[rr] && *+[o][F]{}  \ar@{-}[r]^i& *+[o][F]{}  \ar@{--}[r]^{i+1} & *+[o][F]{}  \ar@{--}[r]^{i+2|\{i,i+2\}} & *+[o][F]{}  \ar@{-}[r]^{i+1}  & *+[o][F]{} \ar@{.}[rr]&& *+[o][F]{}  \ar@{-}[r]^{n-5}  & *+[o][F]{}  \ar@{--}[r]^{n-6} & *+[o][F]{}}$\\

\hline

$(P)\;\xymatrix@-.2pc{ *+[o][F]{}   \ar@{--}[r]^1  &*+[o][F]{}   \ar@{--}[r]^{0|2|\{0,2\}}  & *+[o][F]{} \ar@{--}[r]^1  &  *+[o][F]{} \ar@{-}[r]^0  & *+[o][F]{} \ar@{.}[rr] &&  *+[o][F]{}  \ar@{-}[r]^{n-5} & *+[o][F]{}}$\\

$(Q)\;\xymatrix@-.2pc{ *+[o][F]{}   \ar@{--}[r]^{1|3|\{1,3\}}  &*+[o][F]{}   \ar@{--}[r]^2  & *+[o][F]{} \ar@{--}[r]^1  &  *+[o][F]{} \ar@{-}[r]^0& *+[o][F]{}\ar@{.}[rr] && *+[o][F]{}  \ar@{-}[r]^{n-5} & *+[o][F]{}}$\\

$(R)\; \xymatrix@-.2pc{ *+[o][F]{}   \ar@{-}[r]^1  &*+[o][F]{}   \ar@{-}[r]^0  & *+[o][F]{} \ar@{--}[r]^1  &  *+[o][F]{} \ar@{--}[r]^2 & *+[o][F]{} \ar@{--}[r]^{3|\{1,3\}}& *+[o][F]{} \ar@{-}[r]^2   & *+[o][F]{} \ar@{.}[rr] &&  *+[o][F]{}  \ar@{-}[r]^{n-5} & *+[o][F]{}}$\\

\hline

$ (S)\;\xymatrix@-.2pc{ *+[o][F]{}   \ar@{--}[r]^1  &*+[o][F]{}   \ar@{-}[r]^0  & *+[o][F]{} \ar@{--}[r]^1  &  *+[o][F]{} \ar@{--}[r]^{2|\{0,2\}} & *+[o][F]{} \ar@{-}[r]^1&  *+[o][F]{} \ar@{.}[rr] && *+[o][F]{}  \ar@{-}[r]^{n-5} & *+[o][F]{}}$\\

$(T)\; \xymatrix@-.2pc{ *+[o][F]{}   \ar@{-}[r]^1  &*+[o][F]{}   \ar@{-}[r]^0  & *+[o][F]{} \ar@{--}[r]^1  &  *+[o][F]{} \ar@{-}[r]^2 & *+[o][F]{} \ar@{--}[r]^3& *+[o][F]{} \ar@{--}[r]^{4|\{2,4\}}  &*+[o][F]{} \ar@{-}[r]^3  & *+[o][F]{} \ar@{.}[rr] &&  *+[o][F]{}  \ar@{-}[r]^{n-5} & *+[o][F]{}}$\\

\hline

$(U)\; \xymatrix@-.2pc{ *+[o][F]{}   \ar@{--}[r]^1  &*+[o][F]{}   \ar@{-}[r]^0  & *+[o][F]{} \ar@{-}[r]^1  &  *+[o][F]{} \ar@{--}[r]^2 & *+[o][F]{} \ar@{--}[r]^{3|\{1,3\}}& *+[o][F]{} \ar@{-}[r]^2  & *+[o][F]{} \ar@{.}[rr] && *+[o][F]{}  \ar@{-}[r]^{n-5} & *+[o][F]{}}$\\

\hline

$(V)\, \xymatrix@-.2pc{  *+[o][F]{}   \ar@{--}[r]^1  &*+[o][F]{}   \ar@{-}[r]^0    & *+[o][F]{} \ar@{.}[rr] && *+[o][F]{} \ar@{-}[r]^i& *+[o][F]{} \ar@{--}[r]^{i+1}& *+[o][F]{} \ar@{--}[r]^{i+2|\{i,i+2\}}& *+[o][F]{} \ar@{-}[r]^{i+1} &  *+[o][F]{} \ar@{.}[rr] & & *+[o][F]{}  \ar@{-}[r]^{n-5} & *+[o][F]{}}$\\
\hline
\end{tabular}
\end{tabular}$$
\end{small}
\end{lemma}
\begin{proof}
Let us assume that $\mathcal{F}$ is maximal with respect to the last component.

To describe the sizes of the components of a fracture graph we use the following notation:
$B$ for big component, $T$ for a component with 3 vertices, $D$ for a component with 2 vertices and $O$ for an isolated vertex.

By Proposition~\ref{join} a component of  $\mathcal{F}$ that is at distance one of a big component in  $\mathcal{G}$ is either an isolated vertex or has 3 vertices. Therefore the number of big components is one or two (the last case being possible only when $r=n-4$).
Moreover, the last big component cannot be at distance one from a component with 3 vertices, otherwise we get a contradiction with the maximality of $\mathcal{F}$.
Assume without loss of generality that one of the last two components is big (for, otherwise, we may revert the ordering. For instance if we have the sequence $BDTT$ we can change to $TTDB$). 
There is only one possibility when $r=n-1$ (only one component) and also only one possibility, up to duality, when $r=n-2$ (one big component and an isolated vertex).
For $r=n-3$ we have the following  5 possibilities: $OBO$, $OOB$, $DOB$, $TOB$ and $BOB$. In the list given in this lemma the cases $DOB$, $TOB$ and $BOB$ are put together as $i\geq 0$ can be small. Thus we have only 3 possible graphs when $r=n-3$.

For $r=n-4$, there are 10 possibilities for $\mathcal{F}$. We divide the table of all possibilities in five parts according to the five possibilities for the sizes of the last three components (the ones listed for $r=n-3$).

In total we get 12 possibilities. Some cases have a parameter $i\geq 0$. In any case $r-i$ is the size of a big component, thus
$r-i\geq 4$ or equivalently $i\leq r-4$.
\end{proof}

We now summarize in one single proposition the possibilities we obtained for $\mathcal{G}$ in the previous lemmas.

\begin{proposition}\label{allpossi}
Let $1\leq l \leq 4$, and $n\geq 3+2l$ when $r = n-l$. 
If $\Gamma_j$ are intransitive for all $j\in\{0,\ldots,r-1\}$ and $r\geq n-4$
then $\mathcal{G}$ is, up to duality, one of the graphs given in Table~\ref{tab:allpossi}, with $i\geq 0$ and $r-i\geq 4$.
\end{proposition}

\begin{proof}
When $\mathcal{G}$ is not linear we use Lemmas~\ref{dg=3}, \ref{T} and \ref{G1}. By Lemma~\ref{ic} when $\mathcal{G}$ has a fracture graph that is not linear we have that  $\mathcal{G}=\mathcal{F}$. When all fracture graphs of $\Gamma$ are linear but $\mathcal{G}$ is not linear, we have the possibilities given by Lemma~\ref{G1}, and in that case the fracture graphs are subgraphs of $\mathcal{G}$. Indeed in all cases of Lemma~\ref{G1} there is at least one vertical double edge, otherwise $\mathcal{G}$ admits a fracture graph with a vertex of degree 3. By Lemma~\ref{ic} only vertical edges can be added to those graphs.
All the remaining graphs are linear thus are given by Lemma~\ref{linear}.
\end{proof}

\begin{small}
\begin{table}
$$\begin{tabular}{|c|c|}
\hline
$\xymatrix@-.2pc{ *+[o][F]{}   \ar@{-}[r]^0    & *+[o][F]{} \ar@{.}[r] &  *+[o][F]{}  \ar@{-}[r]^{n-2} & *+[o][F]{}}$
&
$\xymatrix@-.2pc{ *+[o][F]{}   \ar@{-}[r]^1  & *+[o][F]{}  \ar@{-}[r]^0  & *+[o][F]{}  \ar@{.}[r]  & *+[o][F]{}  \ar@{-}[r]^{n-3} & *+[o][F]{}}$\\
\hline
\hline
$(A)\; \xymatrix@-1pc{
     *+[o][F]{}   \ar@{-}[d]_2  \ar@{-}[r]^0     & *+[o][F]{}  \ar@{-}[d]^2  \ar@{-}[r]^1  &   *+[o][F]{}  \ar@{-}[r]^2 &  *+[o][F]{} \ar@{-}[r]^3  &*+[o][F]{}\ar@{.}[rr]&& *+[o][F]{}\ar@{-}[r]^{n-4}&*+[o][F]{}\\
     *+[o][F]{}  \ar@{-}[r]_0     & *+[o][F]{}    &   &&&&&&}$
&

$ (G) \xymatrix@-1pc{
     *+[o][F]{}   \ar@{-}[r]^0  \ar@{=}[d]_2^1     & *+[o][F]{}  \ar@{-}[r]^1 \ar@{-}[d]^2   & *+[o][F]{}  \ar@{.}[rr] && *+[o][F]{}  \ar@{-}[r]^{n-4}& *+[o][F]{}  \\
     *+[o][F]{}  \ar@{-}[r]_0     & *+[o][F]{} &    & & &  }$
\\
\hline
$(K)\; \xymatrix@-1pc{ *+[o][F]{}   \ar@{-}[r]^1  & *+[o][F]{}  \ar@{-}[r]^0    & *+[o][F]{} \ar@{.}[rr] & & *+[o][F]{}  \ar@{-}[r]^{n-4} & *+[o][F]{}  \ar@{-}[r]^{n-5} & *+[o][F]{}}$
&

$(L1)\; \xymatrix@-1pc{ *+[o][F]{}   \ar@{-}[r]^0  & *+[o][F]{}  \ar@{-}[r]^1    &  *+[o][F]{}   \ar@{-}[r]^0   & *+[o][F]{} \ar@{.}[r] &  *+[o][F]{}  \ar@{-}[r]^{n-4} & *+[o][F]{}}$
\\
\hline
$(L2)\; \xymatrix@-1pc{ *+[o][F]{}   \ar@{-}[r]^2  & *+[o][F]{}  \ar@{-}[r]^1    &  *+[o][F]{}   \ar@{-}[r]^0   & *+[o][F]{} \ar@{.}[r] &  *+[o][F]{}  \ar@{-}[r]^{n-4} & *+[o][F]{}}$
&
$(L3)\; \xymatrix@-1pc{ *+[o][F]{}   \ar@{=}[r]^2_0  & *+[o][F]{}  \ar@{-}[r]^1    &  *+[o][F]{}   \ar@{-}[r]^0   & *+[o][F]{} \ar@{.}[r] &  *+[o][F]{}  \ar@{-}[r]^{n-4} & *+[o][F]{}}$
\\
\hline
$ (M1)\; \xymatrix@-1pc{ *+[o][F]{}   \ar@{-}[r]^0    & *+[o][F]{}  \ar@{.}[r] & *+[o][F]{}  \ar@{-}[r]^{i}  & *+[o][F]{}  \ar@{-}[r]^{i+1} & *+[o][F]{}  \ar@{-}[r]^{i+2} & *+[o][F]{}  \ar@{-}[r]^{i+1}  & *+[o][F]{}  \ar@{.}[r] & *+[o][F]{}  \ar@{-}[r]^{n-4}&*+[o][F]{}  \ar@{-}[r]^{n-5}& *+[o][F]{}}$
&
$ (M2)\; \xymatrix@-1pc{ *+[o][F]{}   \ar@{-}[r]^0    & *+[o][F]{}  \ar@{.}[r]&  *+[o][F]{}  \ar@{-}[r]^{i}  & *+[o][F]{}  \ar@{-}[r]^{i+1} & *+[o][F]{}  \ar@{=}[r]^{i}_{i+2} & *+[o][F]{}  \ar@{-}[r]^{i+1}  & *+[o][F]{}  \ar@{.}[r] & *+[o][F]{}  \ar@{-}[r]^{n-4}&*+[o][F]{}  \ar@{-}[r]^{n-5}& *+[o][F]{}}$\\
\hline
\hline
$(B)\; \xymatrix@-1pc{
         *+[o][F]{}   \ar@{-}[d]_3  \ar@{-}[r]^0     & *+[o][F]{}  \ar@{-}[d]^3  \ar@{-}[r]^1  & *+[o][F]{} \ar@{-}[d]^3 \ar@{-}[r]^2 & *+[o][F]{} \ar@{-}[r]^3 & *+[o][F]{} \ar@{-}[r]^4& *+[o][F]{} \ar@{.}[r] &*+[o][F]{} \ar@{-}[r]^{n-5}&*+[o][F]{}\\
      *+[o][F]{}  \ar@{-}[r]_0     & *+[o][F]{}  \ar@{-}[r]_1   &  *+[o][F]{}&  &&&&}$
&
$(C)\; \xymatrix@-1pc{
     &*+[o][F]{}   \ar@{-}[d]_2  \ar@{-}[r]^0     & *+[o][F]{}  \ar@{-}[d]^{2}  \ar@{-}[r]^1  & *+[o][F]{}\ar@{-}[r]^2 & *+[o][F]{}\ar@{-}[r]^3 & *+[o][F]{}\ar@{.}[r]& *+[o][F]{}\ar@{-}[r]^{n-5}&*+[o][F]{}\\
    *+[o][F]{}  \ar@{-}[r]_1     & *+[o][F]{}  \ar@{-}[r]_0     & *+[o][F]{}    &   &&&&}$
\\
\hline
$(D) \; \xymatrix@-1pc{
    *+[o][F]{}  \ar@{-}[r]^1 &  *+[o][F]{}   \ar@{-}[d]_2  \ar@{-}[r]^0     & *+[o][F]{}  \ar@{-}[d]^2  \ar@{-}[r]^1  & *+[o][F]{}\ar@{-}[r]^2 & *+[o][F]{}\ar@{-}[r]^3 & *+[o][F]{}\ar@{.}[r]& *+[o][F]{}\ar@{-}[r]^{n-5}&*+[o][F]{}\\
      & *+[o][F]{}  \ar@{-}[r]_0     & *+[o][F]{}    &   &&&&}$
&
$(E)\; \xymatrix@-1pc{
      *+[o][F]{}   \ar@{-}[d]_2  \ar@{-}[r]^0     & *+[o][F]{}  \ar@{-}[d]^2  \ar@{-}[r]^1  & *+[o][F]{}\ar@{-}[r]^2 & *+[o][F]{}\ar@{-}[r]^3 & *+[o][F]{}\ar@{.}[r]& *+[o][F]{}\ar@{-}[r]^{n-5}&*+[o][F]{}\ar@{-}[r]^{n-6} &*+[o][F]{}\\
      *+[o][F]{}  \ar@{-}[r]_0     & *+[o][F]{}    &   &&&&&&&}$
\\
\hline
$ (H1) \xymatrix@-1pc{
     *+[o][F]{}   \ar@{-}[r]^0  \ar@{=}[d]_3^1    & *+[o][F]{}  \ar@{-}[r]^1 \ar@{-}[d]^3 & *+[o][F]{} \ar@{-}[d]^3  \ar@{-}[r]^2 & *+[o][F]{} \ar@{.}[r]&*+[o][F]{} \ar@{-}[r]^{n-5} &*+[o][F]{}\\
     *+[o][F]{}  \ar@{-}[r]_0    & *+[o][F]{}   \ar@{-}[r]_1 &   *+[o][F]{} & &&}$
&
$ (H2) \xymatrix@-1pc{
*+[o][F]{}   \ar@{-}[r]^0  \ar@{=}[d]_3^2    & *+[o][F]{}  \ar@{-}[r]^1 \ar@{=}[d]^3_2 & *+[o][F]{} \ar@{-}[d]^3  \ar@{-}[r]^2 & *+[o][F]{} \ar@{.}[r]&*+[o][F]{} \ar@{-}[r]^{n-5} &*+[o][F]{}\\
*+[o][F]{}  \ar@{-}[r]_0    & *+[o][F]{}   \ar@{-}[r]_1 &   *+[o][F]{} & &&}$
\\
\hline
$ (I) \xymatrix@-1pc{
     *+[o][F]{}   \ar@{-}[r]^0  \ar@{=}[d]_2^1     & *+[o][F]{}  \ar@{-}[r]^1 \ar@{-}[d]^2 & *+[o][F]{}  \ar@{.}[r] &*+[o][F]{}  \ar@{-}[r]^{n-5} &*+[o][F]{}  \ar@{-}[r]^{n-6} &*+[o][F]{}\\
     *+[o][F]{}  \ar@{-}[r]_0     & *+[o][F]{}    & &&&   }$
&
$(N1)\; \xymatrix@-1pc{ *+[o][F]{}   \ar@{-}[r]^0 & *+[o][F]{} \ar@{-}[r]^1  &  *+[o][F]{} \ar@{-}[r]^0 & *+[o][F]{} \ar@{.}[r] & *+[o][F]{}  \ar@{-}[r]^{n-5}& *+[o][F]{}  \ar@{-}[r]^{n-6}  & *+[o][F]{}}$
\\
\hline
$(N2)\; \xymatrix@-1pc{ *+[o][F]{}   \ar@{-}[r]^2  & *+[o][F]{} \ar@{-}[r]^1  &  *+[o][F]{} \ar@{-}[r]^0 & *+[o][F]{} \ar@{.}[r] & *+[o][F]{}  \ar@{-}[r]^{n-5}& *+[o][F]{}  \ar@{-}[r]^{n-6}  & *+[o][F]{}}$
&
$(N3)\; \xymatrix@-1pc{ *+[o][F]{}   \ar@{=}[r]^0_2  & *+[o][F]{} \ar@{-}[r]^1  &  *+[o][F]{} \ar@{-}[r]^0 & *+[o][F]{} \ar@{.}[r] & *+[o][F]{}  \ar@{-}[r]^{n-5}& *+[o][F]{}  \ar@{-}[r]^{n-6}  & *+[o][F]{}}$
\\
\hline
$(O1)\;\xymatrix@-1pc{ *+[o][F]{}   \ar@{-}[r]^0  &  *+[o][F]{} \ar@{.}[r] & *+[o][F]{}  \ar@{-}[r]^i& *+[o][F]{}  \ar@{-}[r]^{i+1} & *+[o][F]{}  \ar@{-}[r]^{i+2} & *+[o][F]{}  \ar@{-}[r]^{i+1}  & *+[o][F]{} \ar@{.}[r]& *+[o][F]{}  \ar@{-}[r]^{n-5}  & *+[o][F]{}  \ar@{-}[r]^{n-6} & *+[o][F]{}}$
&
$(O2)\;\xymatrix@-1pc{ *+[o][F]{}   \ar@{-}[r]^0  &  *+[o][F]{} \ar@{.}[r] & *+[o][F]{}  \ar@{-}[r]^i& *+[o][F]{}  \ar@{-}[r]^{i+1} & *+[o][F]{}  \ar@{=}[r]^{i+2}_i & *+[o][F]{}  \ar@{-}[r]^{i+1}  & *+[o][F]{} \ar@{.}[r]& *+[o][F]{}  \ar@{-}[r]^{n-5}  & *+[o][F]{}  \ar@{-}[r]^{n-6} & *+[o][F]{}}$
\\
\hline
$(P1)\;\xymatrix@-1pc{ *+[o][F]{}   \ar@{-}[r]^1  &*+[o][F]{}   \ar@{-}[r]^0  & *+[o][F]{} \ar@{-}[r]^1  &  *+[o][F]{} \ar@{-}[r]^0  & *+[o][F]{} \ar@{.}[r] &  *+[o][F]{}  \ar@{-}[r]^{n-5} & *+[o][F]{}}$
&
$(P2)\;\xymatrix@-1pc{ *+[o][F]{}   \ar@{-}[r]^1  &*+[o][F]{}   \ar@{-}[r]^2  & *+[o][F]{} \ar@{-}[r]^1  &  *+[o][F]{} \ar@{-}[r]^0  & *+[o][F]{} \ar@{.}[r] &  *+[o][F]{}  \ar@{-}[r]^{n-5} & *+[o][F]{}}$
\\
\hline
$(P3)\;\xymatrix@-1pc{ *+[o][F]{}   \ar@{-}[r]^1  &*+[o][F]{}   \ar@{=}[r]^0_2  & *+[o][F]{} \ar@{-}[r]^1  &  *+[o][F]{} \ar@{-}[r]^0  & *+[o][F]{} \ar@{.}[r] &  *+[o][F]{}  \ar@{-}[r]^{n-5} & *+[o][F]{}}$
&
$(Q1)\;\xymatrix@-1pc{ *+[o][F]{}   \ar@{-}[r]^1  &*+[o][F]{}   \ar@{-}[r]^2  & *+[o][F]{} \ar@{-}[r]^1  &  *+[o][F]{} \ar@{-}[r]^0& *+[o][F]{}\ar@{.}[r] & *+[o][F]{}  \ar@{-}[r]^{n-5} & *+[o][F]{}}$
\\
\hline
$(Q2)\;\xymatrix@-1pc{ *+[o][F]{}   \ar@{-}[r]^3 &*+[o][F]{}   \ar@{-}[r]^2  & *+[o][F]{} \ar@{-}[r]^1  &  *+[o][F]{} \ar@{-}[r]^0& *+[o][F]{}\ar@{.}[r] & *+[o][F]{}  \ar@{-}[r]^{n-5} & *+[o][F]{}}$
&
$(Q3)\;\xymatrix@-1pc{ *+[o][F]{}   \ar@{=}[r]^1_3  &*+[o][F]{}   \ar@{-}[r]^2  & *+[o][F]{} \ar@{-}[r]^1  &  *+[o][F]{} \ar@{-}[r]^0& *+[o][F]{}\ar@{.}[r] & *+[o][F]{}  \ar@{-}[r]^{n-5} & *+[o][F]{}}$
\\
\hline
$(R1)\; \xymatrix@-1pc{ *+[o][F]{}   \ar@{-}[r]^1  &*+[o][F]{}   \ar@{-}[r]^0  & *+[o][F]{} \ar@{-}[r]^1  &  *+[o][F]{} \ar@{-}[r]^2 & *+[o][F]{} \ar@{-}[r]^3& *+[o][F]{} \ar@{-}[r]^2   & *+[o][F]{} \ar@{.}[r] &  *+[o][F]{}  \ar@{-}[r]^{n-5} & *+[o][F]{}}$
&
$(R2)\; \xymatrix@-1pc{ *+[o][F]{}   \ar@{-}[r]^1  &*+[o][F]{}   \ar@{-}[r]^0  & *+[o][F]{} \ar@{-}[r]^1  &  *+[o][F]{} \ar@{-}[r]^2 & *+[o][F]{} \ar@{=}[r]^3_1& *+[o][F]{} \ar@{-}[r]^2   & *+[o][F]{} \ar@{.}[r] &  *+[o][F]{}  \ar@{-}[r]^{n-5} & *+[o][F]{}}$
\\
\hline
$ (S1)\;\xymatrix@-1pc{ *+[o][F]{}   \ar@{-}[r]^1  &*+[o][F]{}   \ar@{-}[r]^0  & *+[o][F]{} \ar@{-}[r]^1  &  *+[o][F]{} \ar@{-}[r]^2 & *+[o][F]{} \ar@{-}[r]^1&  *+[o][F]{} \ar@{.}[r] & *+[o][F]{}  \ar@{-}[r]^{n-5} & *+[o][F]{}}$
&
$ (S2)\;\xymatrix@-1pc{ *+[o][F]{}   \ar@{-}[r]^1  &*+[o][F]{}   \ar@{-}[r]^0  & *+[o][F]{} \ar@{-}[r]^1  &  *+[o][F]{} \ar@{=}[r]^0_2 & *+[o][F]{} \ar@{-}[r]^1&  *+[o][F]{} \ar@{.}[r] & *+[o][F]{}  \ar@{-}[r]^{n-5} & *+[o][F]{}}$
\\
\hline
$(T1)\; \xymatrix@-1pc{ *+[o][F]{}   \ar@{-}[r]^1  &*+[o][F]{}   \ar@{-}[r]^0  & *+[o][F]{} \ar@{-}[r]^1  &  *+[o][F]{} \ar@{-}[r]^2 & *+[o][F]{} \ar@{-}[r]^3& *+[o][F]{} \ar@{-}[r]^4  &*+[o][F]{} \ar@{-}[r]^3  & *+[o][F]{} \ar@{.}[r] &  *+[o][F]{}  \ar@{-}[r]^{n-5} & *+[o][F]{}}$
&
$(T2)\; \xymatrix@-1pc{ *+[o][F]{}   \ar@{-}[r]^1  &*+[o][F]{}   \ar@{-}[r]^0  & *+[o][F]{} \ar@{-}[r]^1  &  *+[o][F]{} \ar@{-}[r]^2 & *+[o][F]{} \ar@{-}[r]^3& *+[o][F]{} \ar@{=}[r]^4_2  &*+[o][F]{} \ar@{-}[r]^3  & *+[o][F]{} \ar@{.}[r] &  *+[o][F]{}  \ar@{-}[r]^{n-5} & *+[o][F]{}}$
\\
\hline
$(U1)\; \xymatrix@-1pc{ *+[o][F]{}   \ar@{-}[r]^1  &*+[o][F]{}   \ar@{-}[r]^0  & *+[o][F]{} \ar@{-}[r]^1  &  *+[o][F]{} \ar@{-}[r]^2 & *+[o][F]{} \ar@{-}[r]^3& *+[o][F]{} \ar@{-}[r]^2  & *+[o][F]{} \ar@{.}[r] & *+[o][F]{}  \ar@{-}[r]^{n-5} & *+[o][F]{}}$
&
$(U2)\; \xymatrix@-1pc{ *+[o][F]{}   \ar@{-}[r]^1  &*+[o][F]{}   \ar@{-}[r]^0  & *+[o][F]{} \ar@{-}[r]^1  &  *+[o][F]{} \ar@{-}[r]^2 & *+[o][F]{} \ar@{=}[r]^1& *+[o][F]{} \ar@{-}[r]^2  & *+[o][F]{} \ar@{.}[r] & *+[o][F]{}  \ar@{-}[r]^{n-5} & *+[o][F]{}}$
\\
\hline
$(V1)\, \xymatrix@-1pc{  *+[o][F]{}   \ar@{-}[r]^1  &*+[o][F]{}   \ar@{-}[r]^0    & *+[o][F]{} \ar@{.}[r] & *+[o][F]{} \ar@{-}[r]^i& *+[o][F]{} \ar@{-}[r]^{i+1}& *+[o][F]{} \ar@{-}[r]^{i+2}& *+[o][F]{} \ar@{-}[r]^{i+1} &  *+[o][F]{} \ar@{.}[r] &  *+[o][F]{}  \ar@{-}[r]^{n-5} & *+[o][F]{}}$
&
$(V2)\, \xymatrix@-1pc{  *+[o][F]{}   \ar@{-}[r]^1  &*+[o][F]{}   \ar@{-}[r]^0    & *+[o][F]{} \ar@{.}[r] & *+[o][F]{} \ar@{-}[r]^i& *+[o][F]{} \ar@{-}[r]^{i+1}& *+[o][F]{} \ar@{=}[r]^i_{i+2} & *+[o][F]{} \ar@{-}[r]^{i+1} &  *+[o][F]{} \ar@{.}[r] &  *+[o][F]{}  \ar@{-}[r]^{n-5} & *+[o][F]{}}$\\
\hline
\end{tabular}$$
\caption{Possible CPR graphs of rank $\geq n-4$  string C-groups \label{tab:allpossi}}
\end{table}
\end{small}
%%%%%%%%%%%%%%%%%%%%%%%%%%%%%%%%%%%%%%%%
%%%%%%%%%%%%%%%%%%%%%%%%%%%%%%%%%%%%%%%%
%%%%%%%%%%%%%%%%%%%%%%%%%%%%%%%%%%%%%%%%
\section{The Intersection Property}\label{section6}

In Proposition~\ref{allpossi}, 38 families of permutation representation graphs of rank $r\geq n-4$ were constructed for transitive permutation groups of degree $n$.  In what follows we refer to those graphs by their label in Table~\ref{tab:allpossi}.  For each of these graphs we need to check which ones yield string C-groups, i.e., which corresponding groups satisfy the intersection property. 
The result of this process is the classification (up to isomorphism and duality) of all string C-groups of rank $r\geq n-4$ obtained from a permutation group of degree $n$ with only intransitive maximal parabolic subgroups.  The proofs of many of the cases that yield string C-groups are similar; they are often inductive, and all make use of the following two lemmas.

\begin{lemma}\cite{flm}
\label{max}
Let $\Gamma = \langle \rho_0,...,\rho_{r-1} \rangle$ be a sggi.  If  $\Gamma_0:=\langle \rho_1,...,\rho_{r-1} \rangle \textrm{ and } \Gamma_{r-1}:=\langle \rho_0,...,\rho_{r-2} \rangle$ are string C-groups, $\rho_{r-1} \not \in \Gamma_{r-1}$, and $\langle \rho_1,...,\rho_{r-2} \rangle$ is maximal in $\Gamma_0$, then $\Gamma$ is itself a string C-group.
\end{lemma}

\begin{lemma}\cite{flm2}
\label{se0}
If $\Gamma=\left<\rho_i\,|\, i=0,\,\ldots,\,r-1\right>$ and $\Gamma^*=\langle \rho_i \tau^{\eta_i}\,|\, i\in \{0,\,\ldots,\,r-1\} \rangle$ is a sesqui-extension of $\Gamma$ with respect to $\rho_k$, then:
\begin{enumerate}
\item $ \Gamma^* \cong \Gamma$ or $\Gamma^* \cong \Gamma\times \langle \tau\rangle\cong\Gamma\times 2$.
\item if there is an element of $\Gamma$ which is written with an odd number of $\rho_k$'s and is equal to the identity, then $\Gamma^*\cong\Gamma\times \langle \tau\rangle$.
\item if $\Gamma$ is a permutation group, $\tau$ and $\rho_k$ are odd permutations, and all other $\rho_i$ are even permutations, then $\Gamma^*\cong\Gamma$.
\item whenever $\tau \notin \Gamma^*$, $\Gamma$ is a string C-group if and only if $\Gamma^*$ is a string C-group.
\end{enumerate}
\end{lemma}
It is known that the ranks $n-1$ and $n-2$ of Table ~\ref{tab:allpossi} yield string C-groups, we now deal with ranks $n-3$ and $n-4$ separately.
\subsection{Rank $n-3$}
First let us consider the permutation representation graphs from the previous section which have rank $n-3$ with $n\geq 9$.

\begin{proposition}\label{badn-3}
None of the string groups generated by involutions $\Gamma$, described by the permutation representation graphs $(M1)$ nor $(M2, \  i \geq 1)$, are C-groups.
\end{proposition}

\begin{proof}
Let us first consider the group $\Gamma:=\langle \rho_0, \ldots, \rho_{r-1} \rangle$ corresponding to the graph $(M1)$.  For this group consider the intersection $\Gamma_{\leq i+2}$ and $\Gamma_{\geq i+1}$.  If the intersection condition held, then they would intersect in a dihedral group $\langle \rho_{i+1}, \rho_{i+2} \rangle$ of order ten acting on the five points in the support of those two generators.  However, $\Gamma_{\leq i+2}$ and $\Gamma_{\geq i+1}$ are both symmetric groups, and contain the symmetric group acting on these five points.  Therefore $\Gamma_{\leq i+2} \cap \Gamma_{\geq i+1} \cong S_5$, and thus $\Gamma$ is not a C-group.

Now consider the group $(M2)$ with $i\geq 1$. In this case we have $\Gamma_{\leq i+1}\cap\Gamma_{\geq i}\neq \Gamma_{i,i+1}$ by the same reasoning as before. Note that in the case $i=0$ needs to be treated separately as $\Gamma_{\leq i+1}$ is not the symmetric group any more.
\end{proof}

\begin{proposition}\label{L1}
The string group generated by involutions $\Gamma$, described by the permutation representation graph $(L1)$, is a C-group.

\end{proposition}
\begin{proof}
This follows directly from Theorem 3 of~\cite{fl}.
\end{proof}

\begin{proposition}\label{L2}
The sggi $\Gamma$, described by the permutation representation graph $(L2)$, is a C-group.
\end{proposition}
\begin{proof}
We use Lemma~\ref{max} to show that $\Gamma$ is a C-group.
The group $\Gamma_{r-1}$ is assumed to be a C-group by induction.  The group $\Gamma_0$ is a intransitive group acting on two orbits, of size 3 and $n-3$, and thus is a subgroup of $S_3 \times  S_{n-3}$.  Furthermore, it contains a full symmetric group acting on the $n-3$ points in one orbit (as $\Gamma$ acts transitively and contains a transposition on these $n-3$ points).  The element $(\rho_2 \rho_3)^3$ acts trivially on this orbit and as a transposition on the smaller orbit.  The element $(\rho_1 \rho_2 \rho_3)^4$ also acts trivially on the larger orbit, and as a three-cycle on the smaller orbit.  Thus $\Gamma_0 \cong S_3 \times  S_{n-3}$.

To show that $\Gamma_0$ is a C-group, consider the groups $\Gamma_{0,1}$ and $\Gamma_{0,r-1}$.  The group $\Gamma_{0,1}$ is a sesqui-extension of a simplex and thus, by Lemma~\ref{se0}, it is a string C-group isomorphic to $C_2 \times S_{n-4}$ (as $(\rho_2 \rho_3)^3$ acts trivially on the orbit with $n-4$ points).  The group $\Gamma_{0,r-1}$ is a parabolic subgroup of $\Gamma_{r-1}$ and thus is a C-group.  The same arguments that gave the isomorphism type of $\Gamma_0$ show that $\Gamma_{0,r-1} \cong S_3 \times  S_{n-4}$.
Both $\Gamma_{0,r-1}$ and $\Gamma_{0,1}$ are thus C-groups and $\Gamma_{0,1,r-1}$ can similarly be shown to be isomorphic to $C_2 \times S_{n-5}$ which is maximal in $\Gamma_{0,1}$.  Therefore, $\Gamma_0$ is a C-group, and as $\Gamma_{0,r-1}$ is maximal in $\Gamma_0$, $\Gamma$ is a C-group.
\end{proof}
\begin{proposition}\label{L3}
The sggi $\Gamma$, described by the permutation representation graph $(L3)$, is a C-group.
\end{proposition}
\begin{proof}
The groups $\Gamma_0$ and $\Gamma_{0,r-1}$ are the same as in Proposition~\ref{L2}, and thus $\Gamma_0$, is a C-group isomorphic to $S_3 \times S_{n-3}$, and $\Gamma_{0,r-1} \cong S_3 \times S_{n-4}$.   The group $\Gamma_{r-1}$ is assumed to be a C-group by induction.  As $\Gamma_{0,r-1}$ is maximal in $\Gamma_0$, the group $\Gamma$ is a C-group by Lemma~\ref{max}.
\end{proof}

\begin{proposition}\label{M2i=0 }
The sggi $\Gamma$, described by the permutation representation graph $(M2, i=0)$, is a C-group.
\end{proposition}
\begin{proof}
The group $\Gamma_0$ is shown to be a C-group isomorphic to $S_{n-1}$ in Lemma 21 of~\cite{fl}.  The group $\Gamma_{r-1}$ is assumed to be a string C-group by induction.  Finally, the group $\Gamma_{0,r-1}$ is isomorphic to $S_{n-2}$ and thus is maximal in $\Gamma_0$. Hence, by Lemma~\ref{max}, $\Gamma$ is a string C-group.
\end{proof}
\begin{proposition}\label{A}
The sggi $\Gamma$, described by the permutation representation graph $(A)$, is a string C-group.
\end{proposition}
\begin{proof}
The group $\Gamma_0$ is a sesqui-extension of a group which is a string C-group by Lemma 21 of~\cite{fl} and Lemma~\ref{se0}, and thus is itself a C-group.  Furthermore, the element $(\rho_2 \rho_1 \rho_2 \rho_3 \rho_2)^5$ acts as the identity on the larger orbit of $\Gamma_0$ and as a transposition on the smaller orbit.  Therefore $\Gamma_0 \cong C_2 \times S_{n-2}$; similarly, $\Gamma_{0,r-1} \cong C_2 \times S_{n-3}$.  The group $\Gamma_{r-1}$ is assumed to be a C-group by induction.  Therefore, since $\Gamma_{0,r-1}$ is maximal in $\Gamma_0$, $\Gamma$ is a string C-group.
\end{proof}
\begin{proposition}\label{G}
The sggi $\Gamma$, described by the permutation representation graph $(G)$, is a C-group.
\end{proposition}
\begin{proof}
The group $\Gamma_0$ is an intransitive group acting on two orbits of size 2 and $n-2$.  Furthermore, it contains a full symmetric group acting on the larger orbit.  The element $(\rho_2 \rho_1 \rho_2 \rho_3)^3$ acts as the identity on the larger orbit, and as a transposition on the smaller one.  Thus $\Gamma_0 \cong C_2 \times S_{n-2}$; similarly, $\Gamma_{0,r-1} \cong C_2 \times S_{n-3}$.  The group $\Gamma_{0,1}$ is a sesqui-extension of the simplex, and thus is a C-group isomorphic to $C_2 \times S_{n-4}$; similarly, $\Gamma_{0,1,r-1} \cong C_2 \times S_{n-5}$.
We assume that $\Gamma_{r-1}$ is a C-group by induction, and thus $\Gamma_{0,r-1}$ is also a C-group.  Since both $\Gamma_{0,1}$ and $\Gamma_{0,r-1}$ are C-groups, and $\Gamma_{0,1,r-1}$ is maximal in $\Gamma_{0,1}$, the group $\Gamma_0$ is a C-group.
Now as both $\Gamma_0$ and $\Gamma_{r-1}$ are C-groups, and $\Gamma_{0,r-1}$ is maximal in $\Gamma_0$, we conclude that $\Gamma$ is a string C-group.
\end{proof}
\begin{proposition}\label{K}
The sggi $\Gamma$, described by the permutation representation graph $(K)$, is a C-group.
\end{proposition}
\begin{proof}
The groups $\Gamma_0$ and $\Gamma_{r-1}$ are both sesqui-extensions of a group which is a string C-group by Lemma 21 of~\cite{fl}, and thus are themselves C-groups.  Furthermore they are both isomorphic to $C_2 \times S_{n-2}$.  As $\Gamma_{0,r-1} \cong C_2 \times S_{n-4} \times C_2$ is maximal in $\Gamma_0$, the group $\Gamma$ is a string C-group.
\end{proof}
Note that all groups corresponding to the graphs of Table~\ref{tab:allpossi} are transitive and contain a transposition. Thus they are isomorphic to $S_n$.  In Propositions~\ref{badn-3} through~\ref{K} we proved that seven graphs, namely $(A)$, $(G)$, $(K)$, $(L1)$, $(L2)$, $(L3)$ and $M2$ (with $i=0$), correspond to rank $n-3$ string C-groups isomorphic to $S_n$ for $n\geq 9$. So far we have shown that this is the complete list of rank $n-3$ transitive string C-groups with connected diagram when all maximal parabolic subgroups are intransitive. Let us now consider the rank $n-4$ case.
\subsection{Rank $n-4$}

To show which rank $n-4$ groups do not satisfy the intersection property, we could proceed in the manner above, dealing with each rank $n-4$ group ad-hoc.  However, for sake of brevity, we utilize the fact that all the proofs that confirm the actual string C-groups are inductive, with the base case (where $n=11$) checked using \textsc{Magma}~\cite{BCP97}.   Thus, we also use \textsc{Magma} to simplify many of the proofs showing which remaining permutation representation graphs do not yield string C-groups.  In order to do this, we notice that for many of the groups, the base case (where $n=11$) gives a parabolic subgroup of the larger, higher rank, cases.  Thus simply knowing that the case where $n=11$ does not yield a string C-group proves that many of the groups in question are not string C-groups for any $n \geq 11$.  We point out here that although the computer is used to shorten this article, each of the cases can be easily done by hand as well.

\begin{proposition}\label{magmabads}

None of the string groups generated by involutions $\Gamma$, described by the permutation representation graphs $
(C)$, $(D)$, $(H1)$, $(H2)$, $(P2)$, $(P3)$, $(Q1)$, $(Q3)$, $(R1)$, $(R2)$, $(S1)$, $(S2)$, $(T1)$, $(T2)$, $(U1)$, nor $(U2)$, are C-groups.
\end{proposition}
\begin{proof}
It can easily be verified using \textsc{Magma} that when $n=11$ none of these graphs yield a string C-group.  For larger $n$,  let $\Gamma=\langle \rho_0, \ldots, \rho_{n-5} \rangle $ be the group corresponding to the permutation representation graph on $n$ points.  Now consider the group $\Gamma_{\leq 6}$.  This group is exactly one of the groups shown to not satisfy the intersection by \textsc{Magma} above.  Thus, since $\Gamma_{\leq 6}$ is not a string C-group, neither is $\Gamma$.
\end{proof}

Next let us consider the groups given by graphs $(O1)$, $(O2)$, $(V1)$, and $(V2)$.  These cannot be treated as above, as they each represent a family of groups not only indexed by $n$ but also by $i$.

\begin{proposition}\label{badn-4}
None of the string groups generated by involutions $\Gamma$, described by the permutation representation graphs $(O1)$, $(O2, \ i \geq 1)$, $(V1)$, nor $(V2)$, are C-groups.
\end{proposition}

\begin{proof}
Observe that the graphs $(V2)$ with $i=0$ and $(S2)$ are the same and the graphs  $(V2)$ with $i=1$ and $(R2)$ are also the same.
For the remaining cases, the proof of this is the same as the proof of Proposition~\ref{badn-3}. 
\end{proof}

\begin{proposition}\label{B}
The sggi $\Gamma$, described by the permutation representation graph $(B)$, is a C-group.
\end{proposition}
\begin{proof}
We use Lemmas~\ref{max} and \ref{se0} to show that $\Gamma$ is a C-group.  The group $\Gamma_{r-1}$ is assumed to be a C-group by induction.  To show that $\Gamma_0$ is a C-group, we consider the groups $\Gamma_{0,1}$, $\Gamma_{0,1,2}$, $\Gamma_{0,r-1}$, $\Gamma_{0,1,r-1}$, and $\Gamma_{0,1,2,r-1}$.  The groups $\Gamma_{0,1,2}$ and $\Gamma_{0,1,2,r-1}$ are sesqui-extensions of a simplex and thus, by Lemma~\ref{se0}, are both string C-groups isomorphic to $C_2 \times S_{n-6}$ and $C_2 \times S_{n-7}$ respectively.  Note also that $\Gamma_{0,r-1}$ and $\Gamma_{0,1,r-1}$ are string C-groups as they are parabolic subgroups of $\Gamma_{r-1}$.  By Lemma~\ref{max}, since $\Gamma_{0,1,2,r-1}$ is maximal in $\Gamma_{0,1,2}$ and both $\Gamma_{0,1,2}$ and $\Gamma_{0,1,r-1}$ are string C-groups, we conclude that $\Gamma_{0,1}$ is also a string C-group.

The group $\Gamma_{0,1}$ is a sesqui-extension of a string C-group isomorphic to $S_{n-4}$ (see~\cite{fl}).  Furthermore, the element $(\rho_3 \rho_2\rho_3 \rho_4\rho_3)^5$ acts as identity on the orbit of size $n-4$ and thus $\Gamma_{0,1} \cong 2 \times  S_{n-4}$.  Similarly $\Gamma_{0,1,r-1} \cong 2 \times  S_{n-5}$.  Since both $\Gamma_{0,1}$ and $\Gamma_{0,r-1}$  are string C-groups, and $\Gamma_{0,1,r-1}$ is maximal in $\Gamma_{0,1}$, we conclude that $\Gamma_{0}$ is a string C-group.

Finally, the groups $\Gamma_{0}$ and $\Gamma_{0,r-1}$ are sesqui-extensions of string C-groups isomorphic to $S_{n-2}$ (see  $(A)$ in Table~\ref{tab:allpossi}) and $S_{n-3}$, respectively.   Furthermore, the element $(\rho_1 \rho_2\rho_3 \rho_2\rho_3 \rho_4\rho_3)^7$ acts trivially on the larger orbit and as a transposition on the smaller orbit.  Thus $\Gamma_0 \cong 2 \times  S_{n-2}$ and $\Gamma_{0,r-1} \cong 2 \times  S_{n-3}$.  Since both $\Gamma_{0}$ and $\Gamma_{r-1}$  are string C-groups, and $\Gamma_{0,r-1}$ is maximal in $\Gamma_{0}$, we conclude that $\Gamma$ is a string C-group.
\end{proof}

\begin{proposition}\label{E}
The sggi $\Gamma$, described by the permutation representation graph $(E)$, is a C-group.
\end{proposition}
\begin{proof}
To show that $\Gamma$ is a C-group, we consider the groups $\Gamma_{0}$, $\Gamma_{r-1}$, $\Gamma_{0,1}$, $\Gamma_{0,r-1}$, and $\Gamma_{0,1,r-1}$.
The group $\Gamma_{r-1}$ is a sesqui-extension of a string C-group isomorphic to $S_{n-2}$ (see  $(A)$ in Table~\ref{tab:allpossi}) and thus is itself a string C-group by Lemma~\ref{se0}.
Furthermore, the element $(\rho_{r-2} \rho_{r-3})^3$ acts as identity on the larger orbit of $\Gamma_{r-1}$ and thus $\Gamma_{r-1} \cong S_{n-2} \times 2$.  The group $\Gamma_{0,1}$ is a sesqui-extension of a string C-group isomorphic to $S_{n-2}$ (see~\cite{fl})  and is thus a string C-group.   Furthermore, the element $(\rho_2 \rho_3)^3$ acts trivially on the larger orbit of $\Gamma_{0,1}$ and thus $\Gamma_{0,1} \cong S_{n-4} \times 2$.  Similarly, $\Gamma_{0,1,r-1} \cong S_{n-6} \times 2 \times 2$, which is maximal in $\Gamma_{0,1}$.  The group $\Gamma_{0,r-1}$ is also a string C-group as it is a parabolic subgroup of $\Gamma_{r-1}$, and thus by Lemma~\ref{max}, $\Gamma_0$ is a string C-group.  The element $(\rho_2 \rho_1 \rho_2 \rho_3 \rho_2)^5$ acts trivially on the larger orbit of $\Gamma_0$ and thus $\Gamma_0 \cong S_{n-2} \times 2$,  and similarly $\Gamma_{0,r-1}  \cong S_{n-4}\times 2 \times 2$ which is maximal in $\Gamma_0$.  Since both $\Gamma_0$ and $\Gamma_{r-1}$ are string C-groups, and $\Gamma_{0,r-1}$ is maximal in $\Gamma_0$, by Lemma~\ref{max}, $\Gamma$ is a string C-group.
\end{proof}

\begin{proposition}\label{I2}
The sggi $\Gamma$, described by the permutation representation graph $(I2)$, is a C-group.
\end{proposition}
\begin{proof}
To show that $\Gamma$ is a C-group, we consider the groups $\Gamma_{0}$, $\Gamma_{r-1}$, $\Gamma_{0,1}$, $\Gamma_{0,r-1}$, and $\Gamma_{0,1,r-1}$.
The group $\Gamma_{r-1}$ is a sesqui-extension of a string C-group isomorphic to $S_{n-2}$ (see  $(G)$ in Table~\ref{tab:allpossi}) and thus is itself a string C-group by Lemma~\ref{se0}.
Furthermore, the element $(\rho_{r-2} \rho_{r-3})^3$ acts as identity on the larger orbit of $\Gamma_{r-1}$ and thus $\Gamma_{r-1} \cong S_{n-2} \times 2$.
The group $\Gamma_{0,1}$ is a sesqui-extension of a string C-group isomorphic to $S_{n-2}$ (see~\cite{fl}) and is thus a string C-group.   Furthermore, the element $(\rho_2 \rho_3)^3$ acts trivially on the larger orbit of $\Gamma_{0,1}$ and thus $\Gamma_{0,1} \cong S_{n-4} \times 2$.  Similarly, $\Gamma_{0,1,r-1} \cong S_{n-6} \times 2 \times 2$, which is maximal in $\Gamma_{0,1}$.
The group $\Gamma_{0,r-1}$ is also a string C-group as it is a parabolic subgroup of $\Gamma_{r-1}$, and thus by Lemma~\ref{max}, $\Gamma_0$ is a string C-group.  The element $(\rho_2 \rho_1 \rho_2 \rho_3 \rho_4 \rho_3)^3$ acts trivially on the larger orbit of $\Gamma_0$ and thus $\Gamma_0 \cong S_{n-2} \times 2$,  and similarly $\Gamma_{0,r-1}  \cong S_{n-4}\times 2 \times 2$ which is maximal in $\Gamma_0$.  Since both $\Gamma_0$ and $\Gamma_{r-1}$ are string C-groups, and $\Gamma_{0,r-1}$ is maximal in $\Gamma_0$, by Lemma~\ref{max}, $\Gamma$ is a string C-group.
\end{proof}

\begin{proposition}\label{N1}
The sggi $\Gamma$, described by the permutation representation graph $(N1)$, is a C-group.
\end{proposition}
\begin{proof}
To show that $\Gamma$ is a C-group, we consider the groups $\Gamma_{0}$, $\Gamma_{r-1}$, and $\Gamma_{0,r-1}$.
The groups $\Gamma_0$ and $\Gamma_{r-1}$ are both sesqui-extensions of string C-groups which are isomorphic to $S_{n-3}$ and $S_{n-2}$ respectively (see~\cite{fl}), and thus are themselves string C-groups by Lemma~\ref{se0}.
Furthermore, the element $(\rho_{r-2} \rho_{r-3})^3$ acts as identity on the larger orbit of $\Gamma_{r-1}$ and thus $\Gamma_{r-1} \cong S_{n-2} \times 2$.   Similarly $(\rho_{1} \rho_{2})^3$ acts as identity on the larger orbit of $\Gamma_{0}$ and thus $\Gamma_{0} \cong S_{n-3} \times 2$, and $\Gamma_{0,r-1} \cong S_{n-5} \times 2 \times 2$.
Since both $\Gamma_0$ and $\Gamma_{r-1}$ are string C-groups, and $\Gamma_{0,r-1}$ is maximal in $\Gamma_0$, by Lemma~\ref{max}, $\Gamma$ is a string C-group.
\end{proof}

\begin{proposition}\label{N2}
The sggi $\Gamma$, described by the permutation representation graph $(N2)$, is a C-group.
\end{proposition}
\begin{proof}
To show that $\Gamma$ is a C-group, we consider the groups $\Gamma_{0}$, $\Gamma_{r-1}$, $\Gamma_{0,1}$, $\Gamma_{0,r-1}$, and $\Gamma_{0,1,r-1}$.
The group $\Gamma_{r-1}$ is a sesqui-extension of a string C-group isomorphic to $S_{n-2}$ (see  $(L2)$ in Table~\ref{tab:allpossi}) and thus is itself a string C-group by Lemma~\ref{se0}.
Furthermore, the element $(\rho_{r-2} \rho_{r-3})^3$ acts as identity on the larger orbit of $\Gamma_{r-1}$ and thus $\Gamma_{r-1} \cong S_{n-2} \times 2$.
The group $\Gamma_{0,1}$ is a sesqui-extension of a string C-group isomorphic to $S_{n-2}$ (see~\cite{fl}) and is thus a string C-group.   Furthermore, the element $(\rho_2 \rho_3)^3$ acts trivially on the larger orbit of $\Gamma_{0,1}$ and thus $\Gamma_{0,1} \cong S_{n-4} \times 2$.  Similarly, $\Gamma_{0,1,r-1} \cong S_{n-6} \times 2 \times 2$, which is maximal in $\Gamma_{0,1}$.
The group $\Gamma_{0,r-1}$ is also a string C-group as it is a parabolic subgroup of $\Gamma_{r-1}$, and thus by Lemma~\ref{max}, $\Gamma_0$ is a string C-group.
Similar to the proof of Proposition~\ref{L2}, it can be shown that $\Gamma_0 \cong S_{n-3} \times S_3$,  and similarly $\Gamma_{0,r-1}  \cong S_{n-5}\times S_3 \times 2$ which is maximal in $\Gamma_0$.  Since both $\Gamma_0$ and $\Gamma_{r-1}$ are string C-groups, and $\Gamma_{0,r-1}$ is maximal in $\Gamma_0$, by Lemma~\ref{max}, $\Gamma$ is a string C-group.
\end{proof}

\begin{proposition}\label{N3}
The sggi $\Gamma$, described by the permutation representation graph $(N3)$, is a C-group.
\end{proposition}
\begin{proof}
To show that $\Gamma$ is a C-group, we consider the groups $\Gamma_{0}$, $\Gamma_{r-1}$, and $\Gamma_{0,r-1}$.
The group $\Gamma_0$ was shown to be a string C-group isomorphic to $S_3 \times S_{n-3}$ in the previous lemma.  The group $\Gamma_{r-1}$ is a sesqui-extension of a string C-group isomorphic to $S_{n-2}$ (see  $(L3)$ in Table~\ref{tab:allpossi}) and thus is itself a string C-group by Lemma~\ref{se0}.
Furthermore, the element $(\rho_{r-2} \rho_{r-3})^3$ acts as identity on the larger orbit of $\Gamma_{r-1}$ and thus $\Gamma_{r-1} \cong S_{n-2} \times 2$ and $\Gamma_{0,r-1} \cong S_3 \times S_{n-5} \times 2$.
Since both $\Gamma_0$ and $\Gamma_{r-1}$ are string C-groups, and $\Gamma_{0,r-1}$ is maximal in $\Gamma_0$, by Lemma~\ref{max}, $\Gamma$ is a string C-group.
\end{proof}

\begin{proposition}\label{O2i=0}
The sggi $\Gamma$, described by the permutation representation graph $(O2, \ i=0)$, is a C-group.
\end{proposition}
\begin{proof}
To show that $\Gamma$ is a C-group, we consider the groups $\Gamma_{0}$, $\Gamma_{r-1}$, and $\Gamma_{0,r-1}$.
The group $\Gamma_{r-1}$ is a sesqui-extension of a string C-group which is isomorphic to $S_{n-2}$  (see  $(M2, \ i=0)$ in Table~\ref{tab:allpossi}), and thus is itself a string C-group by Lemma~\ref{se0}.
Furthermore, the element $(\rho_{r-2} \rho_{r-3})^3$ acts as identity on the larger orbit of $\Gamma_{r-1}$ and thus $\Gamma_{r-1} \cong S_{n-2} \times 2$.   Similarly, $\Gamma_{0,r-1} \cong S_{n-3} \times 2$.  The group $\Gamma_0$ is a string C-group isomorphic to $S_{n-1}$ (see  $(N1)$ in Table~\ref{tab:allpossi}).
Since both $\Gamma_0$ and $\Gamma_{r-1}$ are string C-groups, and $\Gamma_{0,r-1}$ is maximal in $\Gamma_0$, by Lemma~\ref{max}, $\Gamma$ is a string C-group.
\end{proof}

\begin{proposition}\label{P1}
The sggi $\Gamma$, described by the permutation representation graph $(P1)$, is a C-group.
\end{proposition}
\begin{proof}
To show that $\Gamma$ is a C-group, we consider the groups $\Gamma_{0}$, $\Gamma_{r-1}$, and $\Gamma_{0,r-1}$.
The group $\Gamma_0$ is a sesqui-extension of the rank $n-5$ simplex acting on $n-4$ points, and thus is a string C-group by Lemma~\ref{se0}.  Furthermore, the element $(\rho_{r-2} \rho_{r-3})^3$ acts as identity on the larger orbit of $\Gamma_{0}$ and thus $\Gamma_{0} \cong S_{n-4} \times 2$.  Similarly, $\Gamma_{0,r-1} \cong S_{n-5} \times 2$.
The group $\Gamma_{r-1}$ is assumed to be a C-group by induction (as it too is of type $(P1)$).
Since both $\Gamma_0$ and $\Gamma_{r-1}$ are string C-groups, and $\Gamma_{0,r-1}$ is maximal in $\Gamma_0$, by Lemma~\ref{max}, $\Gamma$ is a string C-group.
\end{proof}

\begin{proposition}\label{Q2}
The sggi $\Gamma$, described by the permutation representation graph $(Q2)$, is a C-group.
\end{proposition}

\begin{proof}
To show that $\Gamma$ is a C-group, we consider the groups $\Gamma_{0}$, $\Gamma_{r-1}$, $\Gamma_{0,1}$, $\Gamma_{0,r-1}$, and $\Gamma_{0,1,r-1}$.

The group $\Gamma_{r-1}$ is assumed to be a string C-group isomorphic to $S_{n-1}$ by induction.
The group $\Gamma_{0,1}$ is was shown to be a C-group isomorphic to $S_{n-4} \times S_3$ in the proof of Proposition~\ref{L2}.  Similarly, $\Gamma_{0,1,r-1} \cong S_{n-5} \times S_3$ which is maximal in $\Gamma_{0,1}$.
The group $\Gamma_{0,r-1}$ is also a string C-group as it is a parabolic subgroup of $\Gamma_{r-1}$, and thus by Lemma~\ref{max}, $\Gamma_0$ is a string C-group.

The group $\Gamma_0$ is a intransitive group acting on two orbits, of size 4 and $n-4$, and thus is a subgroup of $S_4 \times  S_{n-4}$.  Furthermore, it contains a full symmetric group acting on the $n-4$ points in one orbit.  The element $\tau_1:= (\rho_2 \rho_3 \rho_4)^4$ acts trivially on this orbit and as a three cycle on the smaller orbit.  The element $\tau_2:=(\rho_1 \rho_2 \rho_3 \rho_4 )^5$ also acts trivially on the larger orbit, and as a four-cycle on the smaller orbit.   Thus $\langle \tau_1, \tau_2 \rangle \cong S_4$, and therefore $\Gamma_0 \cong S_4 \times  S_{n-4}$.  A similar argument shows that $\Gamma_{0,r-1} \cong S_4 \times  S_{n-5}$, which is maximal in $\Gamma_0$.
Since both $\Gamma_0$ and $\Gamma_{r-1}$ are string C-groups, and $\Gamma_{0,r-1}$ is maximal in $\Gamma_0$, by Lemma~\ref{max}, $\Gamma$ is a string C-group.
\end{proof}
In Propositions~\ref{magmabads} through~\ref{Q2} we proved that nine graphs, namely $(B)$, $(E)$, $(I)$, $(N1)$, $(N2)$, $(N3)$, $(O2)$(with $i=0$), $(P1)$ and $Q2$, correspond to rank $n-4$ string C-groups isomorphic to $S_n$ for $n\geq 11$. So far we have shown that this is the complete  list of rank $n-4$ transitive string C-groups with connected diagram when all maximal parabolic subgroups are intransitive.
%%%%%%%%%%%%%%%%%%%%%%%%%%%%%%%%%%%%%%%%
%%%%%%%%%%%%%%%%%%%%%%%%%%%%%%%%%%%%%%%%

\section{Transitive permutation groups}\label{sectiontrans}

In this section, we prove that, under the hypotheses of Theorem~\ref{maintheorem}, all the maximal parabolic subgroups must be intransitive. This permits to conclude the proof of Theorem~\ref{maintheorem}.
We use the following result, due to Peter J. Cameron and the authors.

\begin{theorem}\cite{CFLM2014}\label{Cameron}
Let $\Gamma$ be a string C-group of rank $r$ which is isomorphic to a
transitive subgroup of $\Sym_n$ other than $\Sym_n$ or $\Alt_n$. Then one of
the following holds:
\begin{enumerate}
\item $r\le n/2$;
\item $n \equiv 2 \mod{4}$, $r = n/2+1$ and $\Gamma$ is $\Cyc_2\wr \Sym_{n/2}$. The generators are
\begin{center}
$\rho_0 = (1,n/2+1)(2,n/2+2)\ldots (n/2,n)$;

$\rho_1 = (2,n/2+2)\ldots (n/2,n)$;

$\rho_i = (i-1,i)(n/2+i-1,n/2+i)$ for $2\leq i \leq n/2$.
\end{center}
Moreover the Sch\"afli type is $[2,3, \ldots, 3,4]$. 
\item $\Gamma$ is transitive imprimitive and is one of the examples appearing in Table~\ref{ploys}.
\item $\Gamma$ is primitive. In this case, $\Gamma$ is obtained from the permutation representation of degree 6 of $\Sym_5 \cong \PGL_2(5)$ and it is the 4-simplex of Sch\"afli type $[3,3,3]$.

\begin{table}
\begin{center}
\begin{tabular}{|c|c|c|c|l|}
\hline
$Degree$&$Number$&Structure&Order&Sch\"afli type\\
\hline
\hline
6&9&$\Sym_3 \times \Sym_3$&36&$[3,2,3]$\\
\hline
6&11&$2^3:\Sym_{3}$&48&$[2,3,3]$\\
6&11&$2^3:\Sym_{3}$&48&$[2,3,4]$\\
\hline
8&45&$2^4:\Sym_{3}:\Sym_{3}$&576&$[3,4,4,3]$\\
\hline
\end{tabular}
\caption{Examples of transitive imprimitive string C-groups of degree $n$ and rank $n/2+1$ for $n\leq 9$.}\label{ploys}
\end{center}
\end{table}

\end{enumerate}

\end{theorem}

\begin{corollary}
Let $1\leq l \leq 4$, and $n\geq 3+2l$ when $r = n-l$. 
Let  $\Gamma$ be a rank $r$ string C-group with connected diagram and isomorphic to a transitive permutation group of degree $n$
Then $\Gamma_i$ is an intransitive group for each $i=0, \ldots, r-1$.
\end{corollary}
\begin{proof}
Case (1) of Theorem~\ref{Cameron} can never happen.
Case (2) may only happen when $n=6$ or $n=10$.
Case(3) and (4) imply that $n = 6$ or $n=8$.
The cases where $n=6$ or $8$ are easily checked with {\sc Magma}.
When $n=10$, $n-4 = n/2 + 1$ and thus, there could exist a string C-group of rank 7 and degree 10 having a maximal parabolic subgroup that is transitive.
If that is the case, $\Gamma$ is a transitive group having a subgroup $\Gamma_i$ of order 3840 as the group of case (2) in Theorem~\ref{Cameron} is $\Cyc_2\wr \Sym_{5}$.
There are 45 transitive groups of degree 10 and only two of them, namely $\Sym_{10}$ and $\Alt_{10}$ may contain subgroups of order 3840. The alternating group is excluded as the first generator of $\Gamma_i$ is composed of five transpositions and hence is an odd permutation. Only $\Sym_{10}$ remains.
All string C-groups representations of rank 7 of $\Sym_{10}$ can be easily computed with {\sc Magma} and checked not to have any maximal parabolic subgroup that is transitive.
\end{proof}

\bibliographystyle{amsplain}

\end{document}